\documentclass[final]{amsart}

\usepackage{booktabs}
\usepackage{fullpage}
\usepackage{tristan}
\usepackage{algorithm}
\usepackage{algpseudocode}
\usepackage{pgfplots}
\usepackage{pgfplotstable}
\usepackage{subcaption}
\usepackage{tikz}
\usetikzlibrary{calc}
\usepackage[giveninits=true,eprint=false,url=false,style=alphabetic]{biblatex}
\author[Calloo, A. et al.]{Ansar Calloo$^{1}$, Matthew Evans $^{2*}$, Fran\c{c}ois Madiot$^{3}$, Tristan Pryer$^{2,4}$}
\address{$^1$Universit\'e Paris-Saclay, CEA, Service de G\'enie Logiciel pour la Simulation, 91191, Gif-sur-Yvette, France}
\address{$^2$Mathematical Sciences, University of Bath}
\address{$^3$Universit\'e Paris-Saclay, CEA, Service d'\'Etudes des R\'eacteurs et de Math\'ematiques Appliqu\'ees, 91191, Gif-sur-Yvette, France}
\address{$^4$Institute of Mathematical Innovation, University of Bath}

\begin{document}

\title{Transport-Matched Penalties for Diffusion Synthetic
  Acceleration of Polytopic Discontinuous Galerkin Discretisations}

\begin{abstract}
  Diffusion synthetic acceleration is most effective when its
  diffusion correction reflects the transport discretisation that
  generates the iteration error. We develop this principle for
  high-order upwind discontinuous Galerkin discretisations of
  discrete-ordinates transport on polytopic meshes. From the discrete
  transport sweep, we derive the exact scalar correction that removes
  the source-iteration scalar error in one step. We prove that the
  associated scalar response is positive and self-adjoint, obtain an
  exact expression for the source-iteration convergence factor, and
  quantify the additional damping produced by vacuum leakage.

  Using the exact correction as a reference, we construct a
  transport-matched modified interior penalty correction whose
  boundary terms are inherited directly from homogeneous vacuum
  inflow. In the optically thick regime, the resulting MIP form
  approximates the exact correction with relative error proportional
  to the effective cell Knudsen number. This gives a strict
  acceleration of source iteration, with bounds uniform in mesh size,
  polynomial degree, and element face count for admissible polytopic
  meshes. Numerical experiments on Cartesian and centroidal Voronoi
  meshes confirm the predicted convergence and correction-operator
  scaling.
\end{abstract}

\maketitle

\section{Introduction}
\label{sec:introduction}

The numerical solution of linear transport equations becomes particularly
challenging in highly scattering, optically thick regimes.  After
discrete-ordinates angular discretisation, source iteration requires
successive directional transport solves coupled through the scalar flux.
In the diffusive regime this iteration becomes increasingly ineffective,
as the dominant error modes become slowly decaying and the convergence
factor approaches unity.

Diffusion synthetic acceleration (DSA) addresses this difficulty by
introducing a diffusion-based correction to the scalar error.  Classical
DSA methods were developed in
\cite{alcouffe1977diffusion,larsen1984diffusion}, and their spectral and
algebraic properties were subsequently analysed in
\cite{ashby1995linear,adams2002fast}.  These works established a central
principle of acceleration methods, the diffusion correction must be
compatible with the discrete transport operator that generates the error.
Stability of the diffusion discretisation alone is not sufficient to
guarantee effective acceleration.

Discontinuous Galerkin (DG) methods provide a natural framework for
transport discretisation, combining local conservation, high-order
approximation, and flexibility on unstructured meshes
\cite{reed1973triangular,lesaint1974finite,wang2009convergence}.  The
general theory of DG methods includes
\cite{cockburn2001runge,cockburn1998local,bassi1997high,
baumann1999discontinuous,riviere2001priori}, while interior penalty
methods for elliptic problems originate from
\cite{babuvska1973finite,baker1977finite,arnold2002unified}.

A substantial body of work has studied the diffusion-limit behaviour
of transport discretisations.  The asymptotic analysis of transport
discretisations was initiated in \cite{larsen1987asymptotic}, with
corresponding results for discontinuous methods in
\cite{borgers1992asymptotic,adams2001thick}.  In particular,
\cite{guermond2010asymptotic} established rigorous conditions under
which upwind discontinuous Galerkin approximations recover the correct
diffusion limit.  These analyses address consistency of the transport
discretisation in the asymptotic regime.  The present work concerns a
different question, the convergence and compatibility of the iterative
solver for the fully discrete transport problem.

DG methods also extend naturally to polygonal and polyhedral meshes.
Relevant developments include polygonal and Voronoi-based discretisations
\cite{wachspress1975rational,ghosh1994voronoi,sukumar2006recent},
polytopic interior penalty methods
\cite{mu2014interior,cangiani2016hp,dong2022robust}, and
agglomeration-based DG methods
\cite{bassi2012flexibility,bassi2014agglomeration}.  Related approaches
include staggered DG methods
\cite{zhao2018staggered}, weak Galerkin methods
\cite{wang2013weak}, virtual element methods
\cite{beirao2013basic}, and multilevel solvers for polytopic interior
penalty discretisations
\cite{antonietti2017multigrid}.  Recent constructions of cycle-free
polytopic transport meshes are given in
\cite{calloo2025cycle}.  A key analytical challenge for such meshes is that
elements may contain many faces with strongly varying sizes.  Estimates
based on individual faces can therefore introduce constants depending on
the local face structure, whereas whole-boundary estimates provide a route
to bounds uniform with respect to the number and relative sizes of faces.

The compatibility of DG transport and diffusion corrections has been
studied extensively.  Discontinuous diffusion corrections were introduced
in \cite{adams1992diffusion}, with compatible schemes for unstructured
meshes and Krylov-accelerated solvers developed in
\cite{warsa2002fully,warsa2004krylov}.  Further developments include
high-order locally refined meshes
\cite{wang2010diffusion}, arbitrary polygonal meshes
\cite{turcksin2014discontinuous}, alternative interior penalty
formulations
\cite{zhou2023revisit}, and heterogeneous media including void regions
\cite{southworth2021heterogeneous}.  The modified interior penalty (MIP)
correction considered here belongs to this framework, combining elliptic
stability requirements with a transport-dependent penalty floor.

The closest analytical result is due to \cite{haut2020diffusion}.
Their analysis establishes compatibility of
interior-penalty-preconditioned scalar corrections for high-order DG
transport on a fixed spatial discretisation, including settings with
curved meshes and cyclic sweep dependencies.  The present work
considers a complementary problem. We identify the exact scalar
correction operator generated by the discrete transport sweep itself
and analyse its approximation by a local MIP correction on admissible
polytopic meshes.  This allows the effect of polynomial degree, mesh
geometry, and physical boundary conditions to be incorporated
explicitly.

The main contributions are as follows.  First, we construct the exact
scalar correction associated with the directional transport solves and
establish its fundamental spectral properties.  This provides a natural
reference operator for analysing diffusion acceleration rather than
assuming a particular diffusion approximation a priori.  Second, we show
how homogeneous vacuum inflow enters this correction through the boundary
treatment of the upwind DG discretisation, and how the resulting boundary
terms are reproduced by the vacuum-matched MIP formulation.  Third, we
prove that the MIP correction is a relative approximation of the exact
correction in the optically thick regime.  Under degree-explicit
whole-boundary trace and inverse estimates, the resulting constants are
uniform with respect to mesh size, polynomial degree, and element face
count.

The present paper is concerned with the analytical structure of the
accelerated iteration and verification of the associated operator
estimates.  A companion computational study
\cite{dsacomp} investigates practical SIP and MIP acceleration over a
broader range of boundary conditions, coefficients, angular quadratures,
meshes, polynomial degrees, and computational costs.  The numerical
experiments presented here therefore focus on the quantities appearing
directly in the analysis, the exact source-iteration factor, the
accelerated contraction estimate, and the relative discrepancy between the
exact and MIP correction operators.

We consider steady monoenergetic transport with isotropic scattering and
constant coefficients on bounded connected polytopal domains in two or
three dimensions.  Homogeneous vacuum inflow is imposed through the zero
exterior trace.  The angular discretisation uses positive centrally paired
quadratures, and the spatial discretisation uses upwind DG methods on
admissible face-to-face polytopic meshes.

The paper is organised as follows.  Section~\ref{sec:transport_pde}
introduces the transport problem and DSA iteration.
Section~\ref{sec:dg_discretisation} defines the DG transport and MIP
forms.  Section~\ref{sec:si_dsa_results} constructs the exact
correction and states the main result.  Section~\ref{sec:proof_si_dsa}
proves the relative correction and contraction estimates.  Finally,
Section~\ref{sec:numerical_verification} gives operator-level
verification of the analytical results.

\section{Transport problem and diffusion correction}
\label{sec:transport_pde}

\subsection{Transport model and diffusive scaling}

Fix a spatial dimension $d\in\{2,3\}$,
and let $\Omega\subset\mathbb R^d$ be a bounded, connected Lipschitz
polytopal domain.  We write $\vec  n$ for the outward unit normal
on $\partial\Omega$.

For $\vec \omega\in\mathbb S^{d-1}$, define the inflow and
outflow boundaries by
\begin{equation}
\label{eq:continuous_inflow_outflow}
\Gamma_\pm(\vec \omega)
:=
\left\{
x\in\partial\Omega:
\pm\vec \omega\cdot\vec  n(x)>0
\right\}.
\end{equation}

For an integrable angular function $v$, define the normalised angular
average
\[
\langle v\rangle
:=
\frac{1}{\lvert\mathbb S^{d-1}\rvert}
\int_{\mathbb S^{d-1}}
v(\vec \omega)\,d\vec \omega.
\]
The normalised spherical moments satisfy
\begin{equation}
\label{eq:continuous_angular_moments}
\langle1\rangle=1,
\qquad
\langle\vec \omega\rangle=\vec 0,
\qquad
\left\langle
\vec \omega\otimes\vec \omega
\right\rangle
=
\frac1dI.
\end{equation}

We consider the steady monoenergetic transport equation
\begin{equation}
\label{eq:continuous_transport}
\begin{aligned}
\vec \omega\cdot\nabla\psi
+\sigma_t^\epsilon\psi
&=
\sigma_s^\epsilon\phi+q
&&\text{in }\Omega,
\\
\psi
&=0
&&\text{on }\Gamma_-(\vec \omega),
\\
\phi
&=\langle\psi\rangle
&&\text{in }\Omega.
\end{aligned}
\end{equation}
Here $\psi$ is the angular flux, $\phi$ is the scalar flux, and $q$ is
an isotropic source.  The boundary condition in
\eqref{eq:continuous_transport} imposes homogeneous vacuum inflow for
each direction $\vec \omega$.

Let $\overline\sigma_t>0$ and $\overline\sigma_a>0$ be constants.  We
use the diffusive scaling
\begin{equation}
\label{eq:continuous_diffusive_scaling}
\sigma_t^\epsilon
=
\frac{\overline\sigma_t}{\epsilon},
\qquad
\sigma_s^\epsilon
=
\frac{\overline\sigma_t}{\epsilon}
-\epsilon\overline\sigma_a,
\qquad
\sigma_a^\epsilon
:=
\sigma_t^\epsilon-\sigma_s^\epsilon
=
\epsilon\overline\sigma_a.
\end{equation}
We assume
\[
0<\epsilon<
\left(
\frac{\overline\sigma_t}{\overline\sigma_a}
\right)^{1/2},
\]
so that $\sigma_s^\epsilon>0$.The scattering ratio is
\begin{equation}
\label{eq:continuous_scattering_ratio}
c_\epsilon
:=
\frac{\sigma_s^\epsilon}{\sigma_t^\epsilon}
=
1-\epsilon^2
\frac{\overline\sigma_a}{\overline\sigma_t}.
\end{equation}

\subsection{Source iteration and diffusion correction}

Given a scalar iterate $\phi^k$, the source-iteration predictor solves
\begin{equation}
\label{eq:continuous_si_predictor}
\begin{aligned}
\vec \omega\cdot\nabla\psi^{k+1/2}
+\sigma_t^\epsilon\psi^{k+1/2}
&=
\sigma_s^\epsilon\phi^k+q
&&\text{in }\Omega,
\\
\psi^{k+1/2}
&=0
&&\text{on }\Gamma_-(\vec \omega),
\\
\phi^{k+1/2}
&=
\left\langle\psi^{k+1/2}\right\rangle
&&\text{in }\Omega.
\end{aligned}
\end{equation}
Let $(\psi^\star,\phi^\star)$ be the transport solution and define
\[
e^k:=\phi^\star-\phi^k,
\qquad
\eta^{k+1/2}:=\psi^\star-\psi^{k+1/2},
\qquad
e^{k+1/2}:=\left\langle\eta^{k+1/2}\right\rangle.
\]
The predictor error satisfies
\[
\vec \omega\cdot\nabla\eta^{k+1/2}
+\sigma_t^\epsilon\eta^{k+1/2}
=
\sigma_s^\epsilon e^k.
\]
Moreover,
\[
\eta^{k+1/2}=0
\qquad
\text{on }\Gamma_-(\vec \omega),
\]
because both the exact solution and the predictor satisfy homogeneous
vacuum inflow.  Introducing
\[
\vec  J^{k+1/2}
:=
\left\langle
\vec \omega\eta^{k+1/2}
\right\rangle,
\qquad
\vec \Pi^{k+1/2}
:=
\left\langle
\vec \omega\otimes\vec \omega
\eta^{k+1/2}
\right\rangle,
\]
the zeroth and first angular moments give
\begin{equation}
\label{eq:continuous_correction_balance}
\nabla\cdot\vec  J^{k+1/2}
+\sigma_a^\epsilon e^{k+1/2}
=
\sigma_s^\epsilon
\qp{e^k-e^{k+1/2}},
\qquad
\sigma_t^\epsilon\vec  J^{k+1/2}
=
-\nabla\cdot\vec \Pi^{k+1/2}.
\end{equation}
In the optically thick regime, the slowly varying error is
approximately isotropic, so
\[
\vec \Pi^{k+1/2}
\approx
\frac1d e^{k+1/2}I,
\qquad
\vec  J^{k+1/2}
\approx
-\epsilon\overline D\nabla e^{k+1/2},
\qquad
\overline D
:=
\frac{1}{d\overline\sigma_t}.
\]
Since
$e^k-e^{k+1/2}=\phi^{k+1/2}-\phi^k$, the moment equations motivate the
interior diffusion correction
\begin{align}
\label{eq:continuous_diffusion_correction}
-\epsilon\overline D\Delta\delta^{k+1}
+\epsilon\overline\sigma_a\delta^{k+1}
&=
\sigma_s^\epsilon
\qp{\phi^{k+1/2}-\phi^k}
\qquad
\text{in }\Omega,
\\
\label{eq:continuous_corrected_update}
\phi^{k+1}
&=
\phi^{k+1/2}+\delta^{k+1}.
\end{align}

For vacuum transport, the moment equations do not by themselves
supply a closed local diffusion boundary condition without an
additional boundary-layer approximation.  We therefore do not impose
a separate Dirichlet or Marshak closure in the analysis.  Instead, the
boundary part of the discrete correction is obtained directly from
the zero exterior trace in the upwind DG transport form.  This gives a
vacuum-matched MIP boundary contribution that can be compared
algebraically with the exact discrete scalar correction.

\section{DG discretisation on polytopic meshes}
\label{sec:dg_discretisation}

\subsection{Mesh, traces, and angular quadrature}

Let $\{\mathcal T_h\}_{h>0}$ be a family of finite, connected,
face-to-face partitions of $\Omega$ into bounded Lipschitz polytopes.
The elements are not assumed to be convex, but the mesh family is
required to satisfy the uniform geometric assumptions stated below.
Denote the sets of interior and physical boundary faces by $\mathcal
E_h^\circ$ and $\mathcal E_h^\partial,$ respectively, and set
$\mathcal E_h := \mathcal E_h^\circ\cup\mathcal E_h^\partial.$

For $K\in\mathcal T_h$, set
\[
h_K:=\operatorname{diam}(K).
\]
For an interior face $e\in\mathcal E_h^\circ$, let $K^+$ and $K^-$
be its adjacent elements and fix a unit normal $\vec n_e$
directed from $K^+$ to $K^-$.  For a boundary face
$e\in\mathcal E_h^\partial$, let $K^+$ be its unique adjacent element
and let $\vec n_e$ be the outward unit normal.

On an interior face, define $\jump{v}:=v^+-v^-,$
$\avg{v}:=\frac12\qp{v^++v^-}.$ On a physical boundary face, we use
the zero exterior trace convention
$v^-:=0.$
Hence
\begin{equation}
\label{eq:vacuum_boundary_jump_average}
\jump{v}=v^+,
\qquad
\avg{v}=\frac12v^+.
\end{equation}
The same convention is used componentwise for vector-valued and
angularly indexed quantities.  In particular,
\[
\avg{\overline D\nabla v\cdot\vec n_e}
=
\frac12
\overline D\nabla v^+\cdot\vec n_e
\qquad
\text{on }e\in\mathcal E_h^\partial.
\]

For every face $e\in\mathcal E_h$, let
\[
\mathcal T_e
:=
\left\{
K\in\mathcal T_h:e\subset\partial K
\right\}.
\]
Thus $\mathcal T_e$ contains two elements on an interior face and one
element on a physical boundary face.

Fix a polynomial degree $p\ge1$ and define
\begin{equation}
\label{eq:dg_space}
V_h
:=
\left\{
v\in L^2(\Omega):
v|_K\in\mathbb P_p(K)
\quad\forall K\in\mathcal T_h
\right\}.
\end{equation}

\begin{assumption}[Admissible mesh family]
\label{ass:admissible_mesh}
There exist constants $C_{\rm inv},C_{\rm tr}>0$,
independent of $K$, $h$, $p$, and the number and relative sizes of the
faces of $K$, such that
\begin{align}
\label{eq:uniform_polynomial_inverse}
\Norm{\nabla v}_{0,K}
&\le
C_{\rm inv}p^2h_K^{-1}\Norm{v}_{0,K},
\\
\label{eq:whole_boundary_trace_inverse}
\Norm{v}_{0,\partial K}^2
&\le
C_{\rm tr}p^2h_K^{-1}\Norm{v}_{0,K}^2
\end{align}
for every $K\in\mathcal T_h$, every $p\ge1$, and every
$v\in\mathbb P_p(K)$.
\end{assumption}

\begin{remark}[Admissible mesh examples]
\label{rem:admissible_mesh_examples}
The assumption is satisfied by the standard shape-regular simplicial
and Cartesian mesh families.  A sufficient geometric condition is that
each element $K$ is star-shaped with respect to a ball
$B(x_K,\rho_K)\subset K$ satisfying
\[
\rho_K\ge\vartheta h_K,
\]
where $\vartheta>0$ is uniform over the mesh family.  Under such a
condition, degree-explicit polynomial inverse and approximation
estimates on polytopic elements follow from
\cite{cangiani2023posteriori,cangiani2022essentially,cangiani2016hp,CangianiGeorgoulisHouston2014}.
Whole-boundary trace estimates with constants independent of the
number and relative sizes of faces follow from the corresponding
geometric covering arguments; see \cite{botti2025trace}.

The essential point is that admissibility is a uniform geometric
property of the elements rather than a restriction to a particular
element topology.  In particular, the condition permits polytopes with
many faces and faces whose diameters are small relative to $h_K$,
provided the element-level covering and star-shapedness constants remain
uniform.  This includes suitable agglomerated meshes and non-convex
elements satisfying these conditions.

The assumption does not include arbitrary polytopic agglomerations.
Although an agglomerated element may have many faces or complicated
geometry, degeneration of the local star-shapedness or covering
properties may destroy the uniform trace and inverse estimates required
for the analysis.

For example, in a Voronoi mesh with separation radius $q$ and covering
radius $R$, the cell generated by $x_i$ satisfies
\[
B(x_i,q/2)\subseteq K_i\subseteq B(x_i,R),
\]
and hence
\[
\frac{r_{K_i}}{h_{K_i}}
\ge
\frac{q}{4R}.
\]
Therefore a uniform bound on $R/q$ gives admissibility of the Voronoi
family.

No quasi-uniformity assumption is imposed at this stage.  However,
strong element degeneration is excluded.  Indeed, taking $v=1$ in
\eqref{eq:whole_boundary_trace_inverse} gives the necessary condition
\[
\frac{h_K|\partial K|}{|K|}
\lesssim1 .
\]
Consequently, families of elements with unbounded aspect ratio are not
covered.
\end{remark}

To account for the degree dependence of the inverse estimates, define
the effective cell Knudsen number by
\begin{equation}
\label{eq:effective_cell_knudsen_number}
\mathrm{Kn}_{h,p}^\epsilon
:=
\max_{K\in\mathcal T_h}
\frac{p^2}{\sigma_t^\epsilon h_K}
=
\max_{K\in\mathcal T_h}
\frac{p^2\epsilon}{\overline\sigma_t h_K}.
\end{equation}
Thus $\mathrm{Kn}_{h,p}^\epsilon\ll1$ means that every cell is
optically thick relative to the polynomial resolution scale.

Let
\[
\left\{
\qp{\vec\omega_m,w_m}
\right\}_{m=1}^{N_\omega}
\subset
\mathbb S^{d-1}\times(0,1)
\]
be a positive angular quadrature satisfying
\begin{equation}
\label{eq:discrete_angular_moments}
\sum_mw_m=1,
\qquad
\sum_mw_m\vec\omega_m=\vec0,
\qquad
\sum_mw_m
\vec\omega_m\otimes\vec\omega_m
=
\frac1dI.
\end{equation}
We also assume central pairing, that is for every $m$, there is an index
$\overline m$ such that
\begin{equation}
\label{eq:central_pairing}
\vec\omega_{\overline m}
=
-\vec\omega_m,
\qquad
w_{\overline m}=w_m.
\end{equation}

For a face $e\in\mathcal E_h$, define the upwind trace in direction
$\vec\omega_m$ by
\[
u_m^{\rm up}
:=
\begin{cases}
u^+,
& \vec\omega_m\cdot\vec n_e\ge0,
\\
u^-,
& \vec\omega_m\cdot\vec n_e<0 .
\end{cases}
\]
On a physical boundary face, this definition uses the interior trace
when $\vec\omega_m\cdot\vec n_e\ge0$ and the zero exterior trace when
$\vec\omega_m\cdot\vec n_e<0$.
Hence the numerical flux imposes homogeneous vacuum inflow.

\subsection{Upwind DG transport form}

For $u,v\in V_h$, define the directional transport form
\begin{equation}
\label{eq:dg_transport_form}
\begin{aligned}
b_{m,h}^\epsilon(u,v)
:={}&
-\sum_{K\in\mathcal T_h}
\qp{u,\vec \omega_m\cdot\nabla v}_K
\\
&+
\sum_{e\in\mathcal E_h}
\left\langle
\qp{\vec \omega_m\cdot\vec  n_e}
u_m^{\rm up},
\jump{v}
\right\rangle_e
+
\qp{\sigma_t^\epsilon u,v}_\Omega.
\end{aligned}
\end{equation}

The angularly averaged upwind dissipation is
\begin{equation}
\label{eq:averaged_jump_form}
j_h(u,v)
:=
\sum_{e\in\mathcal E_h}
\beta_e
\left\langle
\jump{u},\jump{v}
\right\rangle_e,
\end{equation}
where
\begin{equation}
\label{eq:transport_floor}
\beta_e
:=
\frac12
\sum_mw_m
\left|
\vec \omega_m\cdot\vec  n_e
\right|.
\end{equation}

The following identities provide the stability, adjoint symmetry, and
angularly averaged jump dissipation used in the convergence analysis.

\begin{proposition}[Properties of the upwind form]
\label{prop:upwind_properties}
For every $u,v\in V_h$,
\begin{align}
\label{eq:transport_energy_identity}
b_{m,h}^\epsilon(v,v)
&=
\qp{\sigma_t^\epsilon v,v}_\Omega
+
\frac12
\sum_{e\in\mathcal E_h}
\left\langle
\left|
\vec \omega_m\cdot\vec  n_e
\right|
\jump{v},
\jump{v}
\right\rangle_e,
\\
\label{eq:upwind_adjoint_identity}
b_{\overline m,h}^\epsilon(u,v)
&=
b_{m,h}^\epsilon(v,u),
\\
\label{eq:average_transport_identity}
\sum_mw_m b_{m,h}^\epsilon(u,v)
&=
\qp{\sigma_t^\epsilon u,v}_\Omega
+
j_h(u,v).
\end{align}
Moreover,
\begin{equation}
\label{eq:beta_bounds}
\frac{1}{2d}
\le
\beta_e
\le
\frac12.
\end{equation}
In particular, since $\sigma_t^\epsilon>0$, every directional transport
problem is uniquely solvable.
\end{proposition}

The proof is given in Section~\ref{sec:proof_dg_properties}.

\subsection{MIP diffusion form}

Fix $\vartheta\in(0,1)$ and define
\begin{equation}
\label{eq:sip_penalty_no_face_count}
\tau_{{\rm SIP},e}
:=
\frac{C_{\rm tr}p^2}{2\vartheta^2}
\max_{K\in\mathcal T_e}
\frac{\overline D}{h_K}
\end{equation}
The whole-boundary trace estimate implies
\begin{equation}
\label{eq:global_flux_trace_no_face_count}
\sum_{e\in\mathcal E_h}
\tau_{{\rm SIP},e}^{-1}
\Norm{
\avg{
\overline D\nabla v\cdot\vec  n_e
}
}_{0,e}^2
\le
\vartheta^2\overline D
\sum_{K\in\mathcal T_h}
\Norm{\nabla v}_{0,K}^2.
\end{equation}

The MIP penalty is
\begin{equation}
\label{eq:mip_penalty}
\tau_e^\epsilon
:=
\max\left\{
\epsilon\tau_{{\rm SIP},e},
\beta_e
\right\}.
\end{equation}
The first entry in the maximum is the usual SIP penalty for the
physical diffusion coefficient $\epsilon\overline D$ and guarantees
coercivity.  The second entry is the angularly averaged upwind jump
dissipation and provides transport matching.

Define the MIP bilinear form
\begin{equation}
  \label{eq:mip_form}
  \begin{split}
    a_{{\rm MIP},h}^\epsilon(u,v)
    :={}&
    \epsilon
    \sum_{K\in\mathcal T_h}
    \left[
      \qp{\overline D\nabla u,\nabla v}_K
      +
      \qp{\overline\sigma_a u,v}_K
      \right]
    +
    \sum_{e\in\mathcal E_h}
    \tau_e^\epsilon
    \left\langle
    \jump{u},\jump{v}
    \right\rangle_e
    \\  
    &\qquad -
    \epsilon
    \sum_{e\in\mathcal E_h}
    \left\langle
    \avg{
      \overline D\nabla u\cdot\vec  n_e
    },
    \jump{v}
    \right\rangle_e
    -
    \epsilon
    \sum_{e\in\mathcal E_h}
    \left\langle
    \avg{
      \overline D\nabla v\cdot\vec  n_e
    },
    \jump{u}
    \right\rangle_e
  \end{split}
\end{equation}
On a physical boundary face, the contribution of
\eqref{eq:mip_form} is
\begin{align}
\label{eq:vacuum_mip_boundary_contribution}
\sum_{e\in\mathcal E_h^\partial}
\bigg[
&
\tau_e^\epsilon
\left\langle u,v\right\rangle_e
-
\frac{\epsilon}{2}
\left\langle
\overline D\nabla u\cdot\vec  n_e,
v
\right\rangle_e
\nonumber\\
&
-
\frac{\epsilon}{2}
\left\langle
\overline D\nabla v\cdot\vec  n_e,
u
\right\rangle_e
\bigg].
\end{align}
The factors $1/2$ are inherited from the zero exterior average in
\eqref{eq:vacuum_boundary_jump_average}.  They are required for
second-order matching with the exact vacuum transport correction.

This is a transport-matched vacuum boundary form.  It is not the
standard full-flux Nitsche form for a prescribed Dirichlet diffusion
problem, nor is it a Marshak boundary condition.  The comparison
theorem below concerns precisely the form
\eqref{eq:vacuum_mip_boundary_contribution}.

\begin{proposition}[Properties of the MIP form]
\label{prop:mip_properties}
The following statements hold.

\begin{enumerate}
\item
For the homogeneous-vacuum formulation above, every $v\in V_h$
satisfies
\begin{align}
\label{eq:mip_coercivity}
a_{{\rm MIP},h}^\epsilon(v,v)
\ge{}&
(1-\vartheta)
\left[
\epsilon\overline D
\sum_K\Norm{\nabla v}_{0,K}^2
+
\sum_e\tau_e^\epsilon
\Norm{\jump{v}}_{0,e}^2
\right]
\nonumber\\
&+
\epsilon\overline\sigma_a
\Norm{v}_{0,\Omega}^2.
\end{align}
In particular,
$a_{{\rm MIP},h}^\epsilon$ is symmetric positive definite.

\item
If
\begin{equation}
\label{eq:floor_threshold}
\mathrm{Kn}_{h,p}^\epsilon
\le
\mathrm{Kn}_{\rm floor}
:=
\frac{\vartheta^2}{C_{\rm tr}},
\end{equation}
then the transport floor is active on every interior and physical
boundary face:
\begin{equation}
\label{eq:active_transport_floor}
\tau_e^\epsilon=\beta_e
\qquad
\forall e\in\mathcal E_h.
\end{equation}

\end{enumerate}
\end{proposition}

\begin{remark}
For the convex subclass of the admissible elements, the trace constant
can be made explicit.  If $K$ contains a ball of radius $\rho h_K$,
then \cite[Lemma~4.4]{cangiani2022essentially} gives
\[
\Norm{v}_{0,\partial K}^2
\le
\frac{(p+1)(p+d)}{\rho h_K}
\Norm{v}_{0,K}^2
\le
\frac{2(d+1)}{\rho}\,
p^2h_K^{-1}\Norm{v}_{0,K}^2.
\]
Thus one may take
\[
C_{\rm tr}
=
\frac{2(d+1)}{\rho},
\]
and the SIP penalty in
\eqref{eq:sip_penalty_no_face_count} becomes
\[
\tau_{{\rm SIP},e}
=
\frac{(d+1)p^2}{\rho\vartheta^2}
\max_{K\in\mathcal T_e}
\frac{\overline D}{h_K}.
\]
For the more general nonconvex admissible elements described in
Assumption~\ref{ass:admissible_mesh}, the analysis uses the uniform
constant $C_{\rm tr}$ from
\eqref{eq:whole_boundary_trace_inverse}.
\end{remark}

\begin{remark}
The convergence analysis assumes exact applications of the
directional transport inverses and of the MIP inverse.  The term
transport sweep denotes an application of a directional inverse after
any directed dependencies, including directed cycles, have been
resolved.  Cycle construction and practical linear solvers are discussed
in \cite{calloo2025cycle,dsacomp}.
\end{remark}

\section{Source iteration and DSA}
\label{sec:si_dsa_results}

We now put the transport sweep and the MIP correction together.  The
argument has a simple structure.  Source iteration produces a scalar
predictor error.  The exact scalar correction removes that error in
one step.  MIP-DSA applies the same correction equation with the exact
scalar form replaced by the MIP diffusion form.  The main theorem
quantifies this replacement in terms of the effective inverse cell
optical thickness $\mathrm{Kn}_{h,p}^\epsilon$ defined in
\eqref{eq:effective_cell_knudsen_number}.

\subsection{Discrete transport problem and source iteration}

Let $q_h\in V_h$ denote the $L^2$ projection of the isotropic source
$q$ onto $V_h$.  The discrete transport solution consists of
$\psi_{m,h}^\star\in V_h$ and $\phi_h^\star\in V_h$ satisfying
\begin{align}
\label{eq:discrete_transport_problem}
b_{m,h}^\epsilon\qp{\psi_{m,h}^\star,v}
&=
\qp{\sigma_s^\epsilon\phi_h^\star+q_h,v}_\Omega
\qquad
\forall v\in V_h,
\\
\label{eq:discrete_scalar_flux}
\phi_h^\star
&=
\sum_mw_m\psi_{m,h}^\star.
\end{align}

\begin{proposition}[Well-posedness of the discrete transport problem]
\label{prop:discrete_transport_wellposedness}
For every $q_h\in V_h$ and every admissible $\epsilon$, the coupled
problem
\eqref{eq:discrete_transport_problem}--%
\eqref{eq:discrete_scalar_flux}
has a unique solution
\[
\left(
\{\psi_{m,h}^\star\}_{m=1}^{N_\omega},
\phi_h^\star
\right)
\in
V_h^{N_\omega}\times V_h.
\]
\end{proposition}

The proof is given in
Section~\ref{sec:proof_si_dsa}, after the properties of the
normalised scalar response have been established.

Given $\phi_h^k\in V_h$, source iteration computes
$\psi_{m,h}^{k+1/2}\in V_h$ from
\begin{equation}
\label{eq:discrete_si_sweep}
b_{m,h}^\epsilon
\qp{\psi_{m,h}^{k+1/2},v}
=
\qp{\sigma_s^\epsilon\phi_h^k+q_h,v}_\Omega
\qquad
\forall v\in V_h,
\end{equation}
and sets
\begin{equation}
\label{eq:discrete_si_scalar_update}
\phi_h^{k+1/2}
=
\sum_mw_m\psi_{m,h}^{k+1/2}.
\end{equation}
For unaccelerated source iteration,
\[
\phi_h^{k+1}=\phi_h^{k+1/2}.
\]

Define the scalar errors before and after the predictor step by
\[
e_h^k
:=
\phi_h^\star-\phi_h^k,
\qquad
e_h^{k+1/2}
:=
\phi_h^\star-\phi_h^{k+1/2}.
\]
Since the iteration is linear, there is a linear map
$\mathsf G_h^\epsilon:V_h\to V_h$ such that
\begin{equation}
\label{eq:si_error_map}
e_h^{k+1/2}
=
\mathsf G_h^\epsilon e_h^k.
\end{equation}
We call $\mathsf G_h^\epsilon$ the source-iteration error propagator.

\subsection{Exact scalar correction}

To identify the operator approximated by the diffusion correction, we
isolate the angularly averaged response to a common normalised scalar
source.  For $g\in V_h$, let
$\widehat\Psi_{m,h}^\epsilon(g)\in V_h$ satisfy
\begin{equation}
\label{eq:proof_normalised_response}
\epsilon
b_{m,h}^\epsilon
\qp{\widehat\Psi_{m,h}^\epsilon(g),v}
=
\qp{\overline\sigma_t g,v}_\Omega
\qquad
\forall v\in V_h,
\end{equation}
and define the normalised scalar response by
\begin{equation}
\label{eq:proof_normalised_scalar_response}
\mathsf H_h^\epsilon g
:=
\sum_mw_m\widehat\Psi_{m,h}^\epsilon(g).
\end{equation}
Thus $\mathsf H_h^\epsilon$ maps a common normalised source to the
scalar flux produced by the corresponding directional transport
solves.  In particular, $\mathsf H_h^\epsilon$ is the scalar response
operator, or equivalently the discrete transport Schur complement
restricted to the isotropic component.  The exact correction below is
the scalar operator induced by the transport sweep itself.  Classical
diffusion corrections, including the MIP correction studied here, can
then be interpreted as local approximations of this exact operator.

Lemma~\ref{lem:proof_normalised_response} shows that
$\mathsf H_h^\epsilon$ is self-adjoint, positive definite, and hence
invertible.

For a prescribed scalar flux $u\in V_h$,
$(\mathsf H_h^\epsilon)^{-1}u$ is the normalised common source
required to produce $u$.  The scattering source associated with $u$
is $c_\epsilon u$.  This motivates the exact scalar correction form
\begin{equation}
\label{eq:exact_correction_form_results}
a_{{\rm ex},h}^\epsilon(u,v)
:=
\frac1\epsilon
\qp{
\overline\sigma_t
\left[
\qp{\mathsf H_h^\epsilon}^{-1}u
-c_\epsilon u
\right],
v
}_\Omega.
\end{equation}

Subtracting the source-iteration sweep from the discrete transport
problem and multiplying by $\epsilon$ gives
\[
e_h^{k+1/2}
=
c_\epsilon\mathsf H_h^\epsilon e_h^k.
\]
Consequently,
\[
\qp{\mathsf H_h^\epsilon}^{-1}e_h^{k+1/2}
=
c_\epsilon e_h^k,
\]
and hence
\begin{equation}
\label{eq:exact_correction_identity_results}
a_{{\rm ex},h}^\epsilon
\qp{e_h^{k+1/2},v}
=
\qp{
\sigma_s^\epsilon
\qp{e_h^k-e_h^{k+1/2}},
v
}_\Omega
\qquad
\forall v\in V_h.
\end{equation}
Thus solving the exact correction equation with the
source-iteration defect on the right-hand side returns
$e_h^{k+1/2}$ itself and removes the scalar error in one step.

The form $a_{{\rm ex},h}^\epsilon$ is symmetric positive definite and
supplies the natural energy in which to compare source iteration and
MIP--DSA. Define
\begin{equation}
\label{eq:exact_inner_product}
\Norm{v}_{\rm ex}^2
:=
a_{{\rm ex},h}^\epsilon(v,v).
\end{equation}
For a linear map $\mathsf A:V_h\to V_h$, define
\[
\Norm{\mathsf A}_{\rm ex}
:=
\sup_{0\ne v\in V_h}
\frac{\Norm{\mathsf Av}_{\rm ex}}{\Norm{v}_{\rm ex}}.
\]
We write $\rho(\mathsf A)$ for the spectral radius of $\mathsf A$.

\subsection{MIP-DSA iteration}

MIP-DSA replaces the exact scalar correction form by the local
diffusion form.  Given the predictor $\phi_h^{k+1/2}$, compute
$\delta_h^{k+1}\in V_h$ from
\begin{equation}
\label{eq:mip_dsa_problem_results}
a_{{\rm MIP},h}^\epsilon
\qp{\delta_h^{k+1},v}
=
\qp{
\sigma_s^\epsilon
\qp{\phi_h^{k+1/2}-\phi_h^k},v
}_\Omega
\qquad
\forall v\in V_h,
\end{equation}
and update
\begin{equation}
\label{eq:mip_dsa_update_results}
\phi_h^{k+1}
=
\phi_h^{k+1/2}+\delta_h^{k+1}.
\end{equation}

Since
\[
\phi_h^{k+1/2}-\phi_h^k
=
e_h^k-e_h^{k+1/2},
\]
the right-hand sides of
\eqref{eq:exact_correction_identity_results} and
\eqref{eq:mip_dsa_problem_results} are identical.  MIP-DSA therefore
makes the replacement
\[
a_{{\rm ex},h}^\epsilon
\quad\longrightarrow\quad
a_{{\rm MIP},h}^\epsilon.
\]

This observation gives the error equation that drives the analysis:
\begin{equation}
\label{eq:mip_correction_exact_rhs}
\begin{aligned}
a_{{\rm MIP},h}^\epsilon
\qp{\delta_h^{k+1},v}
&=
a_{{\rm ex},h}^\epsilon
\qp{e_h^{k+1/2},v}
\qquad
\forall v\in V_h,
\\
e_h^{k+1}
&=
e_h^{k+1/2}-\delta_h^{k+1}.
\end{aligned}
\end{equation}
and therefore
\begin{equation}
\label{eq:dsa_defect_equation}
a_{{\rm MIP},h}^\epsilon
\qp{e_h^{k+1},v}
=
a_{{\rm MIP},h}^\epsilon
\qp{e_h^{k+1/2},v}
-
a_{{\rm ex},h}^\epsilon
\qp{e_h^{k+1/2},v}
\qquad
\forall v\in V_h.
\end{equation}
Thus the corrected error is controlled directly by the difference
between the MIP and exact scalar correction forms.

Let $\mathsf E_h^\epsilon:V_h\to V_h$ denote the resulting DSA error
propagator:
\begin{equation}
\label{eq:dsa_error_map}
e_h^{k+1}
=
\mathsf E_h^\epsilon e_h^k.
\end{equation}

\subsection{Main theorem}

For later reference, define the MIP energy norm by
\begin{equation}
\label{eq:mip_energy_norm}
\Norm{v}_{\rm MIP}^2
:=
a_{{\rm MIP},h}^\epsilon(v,v).
\end{equation}
The key estimate is a relative comparison of the exact and MIP
correction forms.  Its size is proportional to
$\mathrm{Kn}_{h,p}^\epsilon$, the inverse cell optical thickness
measured relative to the polynomial resolution scale.  Equation
\eqref{eq:dsa_defect_equation} then converts this form comparison into
the DSA contraction estimate.

\begin{theorem}[Source iteration and MIP-DSA in the optically thick regime]
\label{thm:main_si_dsa}
Let $d\in\{2,3\}$ and $p\ge1$.  Consider the homogeneous-vacuum DG
discretisation above on a family of polytopic meshes with nonempty
physical boundary
$\mathcal E_h^\partial$.  On each physical boundary face, impose
vacuum inflow through the zero exterior trace.  Assume that the
polynomial inverse estimate
\eqref{eq:uniform_polynomial_inverse} and the whole-boundary trace
inverse estimate
\eqref{eq:whole_boundary_trace_inverse} hold with constants independent
of $K$, $h$, $p$, and the number of faces of $K$.  Assume also that
the angular quadrature satisfies the positivity, moment, and
central-pairing conditions stated above.

The conclusions concerning MIP--DSA below apply in the optically thick
regime where the transport floor is active.

The scalar source-iteration error propagator is self-adjoint in the
exact correction inner product and has the exact contraction factor
\begin{equation}
\label{eq:main_si_rate}
\Norm{\mathsf G_h^\epsilon}_{\rm ex}
=
\rho\qp{\mathsf G_h^\epsilon}
=
c_\epsilon
\lambda_{\max}\qp{\mathsf H_h^\epsilon}
\le
1-\epsilon^2
\frac{\overline\sigma_a}{\overline\sigma_t}.
\end{equation}
Under homogeneous vacuum inflow,
$\lambda_{\max}(\mathsf H_h^\epsilon)<1$, so the source-iteration
factor is strictly smaller than the coefficient-only bound
$c_\epsilon$.

There are constants $C>0$ and $\mathrm{Kn}_0>0$, depending only on
$d$, the fixed coefficients, $\vartheta$, the angular quadrature, and
the uniform inverse-estimate constants, such that
\[
\mathrm{Kn}_0
\le
\min\left\{
\mathrm{Kn}_{\rm floor},
\frac{1}{2C}
\right\}.
\]
The constants are independent of $h$, $\epsilon$, $p$, and the number
of element faces.  If
\begin{equation}
\label{eq:main_optically_thick_condition}
\mathrm{Kn}_{h,p}^\epsilon
\le
\mathrm{Kn}_0,
\end{equation}
then the transport floor is active,
\[
\tau_e^\epsilon=\beta_e
\qquad
\forall e\in\mathcal E_h,
\]
and the MIP correction form is symmetric positive definite.  On
physical boundary faces it contains the vacuum-matched contribution
\eqref{eq:vacuum_mip_boundary_contribution}.  Moreover, the exact and MIP
correction forms satisfy the relative estimate
\begin{equation}
\label{eq:main_form_comparison}
\norm{
a_{{\rm ex},h}^\epsilon(u,v)
-
a_{{\rm MIP},h}^\epsilon(u,v)
}
\le
C\,\mathrm{Kn}_{h,p}^\epsilon
\Norm{u}_{\rm MIP}
\Norm{v}_{\rm MIP}
\qquad
\forall u,v\in V_h.
\end{equation}
Consequently, within the optically thick regime
\eqref{eq:main_optically_thick_condition}, the MIP--DSA error
propagator satisfies
\begin{equation}
\label{eq:main_dsa_rate}
\Norm{\mathsf E_h^\epsilon}_{\rm ex}
\le
C\,\mathrm{Kn}_{h,p}^\epsilon\,
\rho\qp{\mathsf G_h^\epsilon},
\end{equation}
and hence
\begin{equation}
\label{eq:main_strict_comparison}
\rho\qp{\mathsf E_h^\epsilon}
\le
\Norm{\mathsf E_h^\epsilon}_{\rm ex}
\le
\frac12\rho\qp{\mathsf G_h^\epsilon}
<
\Norm{\mathsf G_h^\epsilon}_{\rm ex}
=
\rho\qp{\mathsf G_h^\epsilon}.
\end{equation}
Thus MIP--DSA is a strict acceleration of source iteration in the
optically thick regime covered by the relative correction estimate.
\end{theorem}

\begin{remark}[Interpretation of the contraction estimate]
Condition~\eqref{eq:main_optically_thick_condition} is equivalent to
the uniform lower bound
\[
\min_{K\in\mathcal T_h}
\frac{\sigma_t^\epsilon h_K}{p^2}
\ge
\frac{1}{\mathrm{Kn}_0}.
\]
Thus every cell must be optically thick on the polynomial inverse
scale $h_K/p^2$.

The two propagators are measured in the same exact-correction energy,
so their contraction factors can be compared directly.  In
particular,
\begin{equation}
\label{eq:dsa_si_factor_ratio}
\frac{
\Norm{\mathsf E_h^\epsilon}_{\rm ex}
}{
\Norm{\mathsf G_h^\epsilon}_{\rm ex}
}
\le
C\mathrm{Kn}_{h,p}^\epsilon.
\end{equation}
The corresponding error estimates are
\begin{align*}
\Norm{e_{h,\mathrm{SI}}^k}_{\rm ex}
&\le
\rho\qp{\mathsf G_h^\epsilon}^k
\Norm{e_h^0}_{\rm ex},
\\
\Norm{e_{h,\mathrm{DSA}}^k}_{\rm ex}
&\le
\left(
C\mathrm{Kn}_{h,p}^\epsilon
\rho\qp{\mathsf G_h^\epsilon}
\right)^k
\Norm{e_h^0}_{\rm ex}.
\end{align*}
Since
$\rho(\mathsf G_h^\epsilon)\le c_\epsilon$, the previous bounds with
$c_\epsilon$ remain valid as upper bounds, but they are no longer
equalities in the vacuum case.
Moreover, the relative form estimate gives
\[
\left(
1-C\mathrm{Kn}_{h,p}^\epsilon
\right)
\Norm{v}_{\rm MIP}^2
\le
\Norm{v}_{\rm ex}^2
\le
\left(
1+C\mathrm{Kn}_{h,p}^\epsilon
\right)
\Norm{v}_{\rm MIP}^2,
\]
so the exact-correction and MIP energies are uniformly equivalent
under~\eqref{eq:main_optically_thick_condition}.

For fixed $h$ and $p$,
$\mathrm{Kn}_{h,p}^\epsilon=O(\epsilon)$.  Moreover, on the fixed
finite-dimensional space,
\[
\mathsf H_h^\epsilon\longrightarrow I
\qquad
\text{as }\epsilon\to0.
\]
Consequently,
\[
\rho\qp{\mathsf G_h^\epsilon}
=
c_\epsilon
\lambda_{\max}\qp{\mathsf H_h^\epsilon}
\longrightarrow1 .
\]
Thus source iteration becomes increasingly ineffective in the
diffusive limit, whereas the MIP--DSA estimate predicts a contraction
factor controlled by the effective cell Knudsen number.
\end{remark}

\section{Proofs of the discrete results}
\label{sec:proof_si_dsa}

Throughout this section, constants denoted by $C$, possibly with
subscripts, are independent of $h$, $\epsilon$, $p$, and the number
of faces of an element.  They may depend on $d$, the fixed
coefficients, the angular quadrature, $\vartheta$, and the constants
in
\eqref{eq:uniform_polynomial_inverse}--%
\eqref{eq:whole_boundary_trace_inverse}.

\subsection{Proofs of the DG properties}
\label{sec:proof_dg_properties}

\begin{proof}[Proof of Proposition~\ref{prop:upwind_properties}]
On every interior or boundary face, the upwind flux satisfies
\[
\qp{\vec \omega_m\cdot\vec  n_e}
u_m^{\rm up}
=
\qp{\vec \omega_m\cdot\vec  n_e}
\avg{u}
+
\frac12
\left|
\vec \omega_m\cdot\vec  n_e
\right|
\jump{u}.
\]
On a physical boundary face this identity uses the zero exterior trace.
Consequently,
\begin{align}
\label{eq:proof_central_upwind_decomposition}
b_{m,h}^\epsilon(u,v)
={}&
-\sum_K
\qp{u,\vec \omega_m\cdot\nabla v}_K
+
\sum_e
\left\langle
\qp{\vec \omega_m\cdot\vec  n_e}
\avg{u},
\jump{v}
\right\rangle_e
\nonumber\\
&+
\frac12
\sum_e
\left\langle
\left|
\vec \omega_m\cdot\vec  n_e
\right|
\jump{u},
\jump{v}
\right\rangle_e
+
\qp{\sigma_t^\epsilon u,v}_\Omega.
\end{align}
Taking $u=v$ and integrating elementwise gives
\[
-\sum_K
\qp{v,\vec \omega_m\cdot\nabla v}_K
=
-\sum_e
\left\langle
\qp{\vec \omega_m\cdot\vec  n_e}
\avg{v},
\jump{v}
\right\rangle_e.
\]
The identity includes the physical boundary because
$\avg{v}=v/2$ and $\jump{v}=v$ there.  The central terms therefore
cancel, proving \eqref{eq:transport_energy_identity}.

Changing $\vec \omega_m$ to
$-\vec \omega_m$ reverses the central part of
\eqref{eq:proof_central_upwind_decomposition} and leaves its symmetric
dissipation and collision terms unchanged.  Elementwise integration
by parts then gives \eqref{eq:upwind_adjoint_identity}.  Pairing $m$
with $\overline m$ and using $w_{\overline m}=w_m$ cancels the angular
average of the central part and proves
\eqref{eq:average_transport_identity}.

For the bounds on $\beta_e$,
\[
2\beta_e
=
\sum_mw_m
\left|
\vec \omega_m\cdot\vec  n_e
\right|.
\]
Since $|x|\ge x^2$ for $|x|\le1$, the second-moment condition gives
\[
2\beta_e
\ge
\sum_mw_m
\qp{\vec \omega_m\cdot\vec  n_e}^2
=
\frac1d.
\]
Thus $\beta_e\ge1/(2d)$.  The upper bound follows from
$|\vec \omega_m\cdot\vec  n_e|\le1$ and
$\sum_mw_m=1$.

Finally, the right-hand side of
\eqref{eq:transport_energy_identity} is positive for every nonzero
$v\in V_h$.  Hence every directional transport problem is uniquely
solvable.
\end{proof}

\begin{proof}[Proof of Proposition~\ref{prop:mip_properties}]
We first verify the whole-boundary flux estimate.  For every
$e\in\mathcal E_h$,
\[
\Norm{
\avg{\overline D\nabla v\cdot\vec  n_e}
}_{0,e}^2
\le
\frac{\overline D^2}{2}
\sum_{K\in\mathcal T_e}
\Norm{
\nabla v|_K\cdot\vec  n_e
}_{0,e}^2.
\]
For a boundary face, the left-hand side contains the additional factor
$1/4$ from
\eqref{eq:vacuum_boundary_jump_average}, so the displayed estimate
continues to hold.

By \eqref{eq:sip_penalty_no_face_count}, for every
$K\in\mathcal T_e$,
\[
\tau_{{\rm SIP},e}^{-1}
\frac{\overline D^2}{2}
\le
\frac{\vartheta^2\overline D}
{C_{\rm tr}p^2}
h_K.
\]
Summing first over the faces of each element and then applying
\eqref{eq:whole_boundary_trace_inverse} componentwise to $\nabla v$,
whose components belong to
$\mathbb P_{p-1}(K)\subset\mathbb P_p(K)$, proves
\eqref{eq:global_flux_trace_no_face_count}.

Since
$\tau_e^\epsilon\ge\epsilon\tau_{{\rm SIP},e}$,
Cauchy--Schwarz and
\eqref{eq:global_flux_trace_no_face_count} give
\begin{align*}
&\epsilon
\left|
\sum_e
\left\langle
\avg{\overline D\nabla v\cdot\vec  n_e},
\jump{v}
\right\rangle_e
\right|
\\
&\qquad\le
\epsilon
\left(
\sum_e\tau_{{\rm SIP},e}^{-1}
\Norm{
\avg{\overline D\nabla v\cdot\vec  n_e}
}_{0,e}^2
\right)^{1/2}
\left(
\sum_e\tau_{{\rm SIP},e}
\Norm{\jump{v}}_{0,e}^2
\right)^{1/2}
\\
&\qquad\le
\vartheta
\left(
\epsilon\overline D
\sum_K\Norm{\nabla v}_{0,K}^2
\right)^{1/2}
\left(
\sum_e\tau_e^\epsilon
\Norm{\jump{v}}_{0,e}^2
\right)^{1/2}.
\end{align*}
Young's inequality therefore yields
\[
2\epsilon
\left|
\sum_e
\left\langle
\avg{\overline D\nabla v\cdot\vec  n_e},
\jump{v}
\right\rangle_e
\right|
\le
\vartheta\epsilon\overline D
\sum_K\Norm{\nabla v}_{0,K}^2
+
\vartheta
\sum_e\tau_e^\epsilon
\Norm{\jump{v}}_{0,e}^2.
\]
Substitution into \eqref{eq:mip_form} proves
\eqref{eq:mip_coercivity}.

Finally,
$\overline D=(d\overline\sigma_t)^{-1}$ gives
\[
\epsilon\tau_{{\rm SIP},e}
\le
\frac{C_{\rm tr}}{2d\vartheta^2}
\max_{K\in\mathcal T_e}
\frac{p^2\epsilon}{\overline\sigma_t h_K}
\le
\frac{C_{\rm tr}}{2d\vartheta^2}
\mathrm{Kn}_{h,p}^\epsilon.
\]
Under \eqref{eq:floor_threshold}, this is at most $1/(2d)$.
Equation \eqref{eq:beta_bounds} therefore gives
$\epsilon\tau_{{\rm SIP},e}\le\beta_e$ on every interior or boundary
face, proving \eqref{eq:active_transport_floor}.

\end{proof}

\subsection{Scalar transport response and source iteration}

Define the collision forms
\begin{equation}
\label{eq:proof_collision_forms}
m_t(u,v)
:=
\qp{\overline\sigma_tu,v}_\Omega,
\qquad
m_a(u,v)
:=
\qp{\overline\sigma_au,v}_\Omega,
\end{equation}
and write
\[
\Norm{v}_t^2:=m_t(v,v).
\]
Since the coefficients are constant,
\begin{equation}
\label{eq:proof_absorption_ratio}
m_a(u,v)
=
\frac{\overline\sigma_a}{\overline\sigma_t}
m_t(u,v).
\end{equation}

\begin{lemma}[Normalised scalar response]
\label{lem:proof_normalised_response}
The map $\mathsf H_h^\epsilon$ is self-adjoint and positive definite
in the $m_t$ inner product, and
\begin{equation}
\label{eq:proof_H_bounds}
0\prec\mathsf H_h^\epsilon\preceq I.
\end{equation}
Let
\[
\lambda_{\max}^\epsilon
:=
\lambda_{\max}\qp{\mathsf H_h^\epsilon}.
\]
Then
\[
0<\lambda_{\max}^\epsilon\le1.
\]
If $\mathcal E_h^\partial\ne\varnothing$ and homogeneous vacuum inflow
is imposed, then
\[
\mathsf H_h^\epsilon\prec I,
\qquad
\lambda_{\max}^\epsilon<1.
\]
Consequently, $a_{{\rm ex},h}^\epsilon$ is symmetric positive
definite.
\end{lemma}

\begin{proof}
Write
$\widehat\Psi_m=\widehat\Psi_{m,h}^\epsilon(g)$.
Testing \eqref{eq:proof_normalised_response} with
$v=\widehat\Psi_m$ and using
\eqref{eq:transport_energy_identity} gives
\[
m_t(g,\widehat\Psi_m)
=
\Norm{\widehat\Psi_m}_t^2
+
\frac{\epsilon}{2}
\sum_e
\left\langle
\left|
\vec \omega_m\cdot\vec  n_e
\right|
\jump{\widehat\Psi_m},
\jump{\widehat\Psi_m}
\right\rangle_e
\ge
\Norm{\widehat\Psi_m}_t^2.
\]
Cauchy--Schwarz therefore gives
\[
\Norm{\widehat\Psi_m}_t
\le
\Norm{g}_t.
\]
Moreover,
\begin{align*}
m_t\qp{\mathsf H_h^\epsilon g,g}
&=
\sum_mw_m m_t\qp{\widehat\Psi_m,g}
\\
&=
\sum_mw_m
\left[
\Norm{\widehat\Psi_m}_t^2
+
\frac{\epsilon}{2}
\sum_e
\left\langle
\left|
\vec \omega_m\cdot\vec  n_e
\right|
\jump{\widehat\Psi_m},
\jump{\widehat\Psi_m}
\right\rangle_e
\right]
>0
\end{align*}
for $g\ne0$, while
\[
m_t\qp{\mathsf H_h^\epsilon g,g}
\le
\sum_mw_m
\Norm{\widehat\Psi_m}_t\Norm{g}_t
\le
\Norm{g}_t^2.
\]
The adjoint identity \eqref{eq:upwind_adjoint_identity} shows that the
response in direction $\overline m$ is the $m_t$-adjoint of the
response in direction $m$. Central pairing therefore makes
$\mathsf H_h^\epsilon$ self-adjoint. The preceding quadratic estimate
then proves the Loewner bounds in \eqref{eq:proof_H_bounds}.

Suppose first that
$\mathcal E_h^\partial\ne\varnothing$ and that
\[
\mathsf H_h^\epsilon g=g
\]
for some $g\in V_h$.  The preceding estimates then give
\begin{align*}
\Norm{g}_t^2
&=
m_t\qp{\mathsf H_h^\epsilon g,g}
\\
&=
\sum_mw_m
m_t\qp{g,\widehat\Psi_{m,h}^\epsilon(g)}
\\
&\le
\sum_mw_m
\Norm{g}_t
\Norm{\widehat\Psi_{m,h}^\epsilon(g)}_t
\\
&\le
\Norm{g}_t^2.
\end{align*}
Hence equality holds at every step.  Since every quadrature weight is
positive, equality holds separately for every ordinate.  Equality in
Cauchy--Schwarz and
\[
\Norm{\widehat\Psi_{m,h}^\epsilon(g)}_t
\le
\Norm{g}_t
\]
therefore imply
\[
\widehat\Psi_{m,h}^\epsilon(g)=g
\qquad
\text{for every }m.
\]
Equality in the directional energy identity also gives
\[
\sum_{e\in\mathcal E_h}
\left\langle
\left|
\vec \omega_m\cdot\vec  n_e
\right|
\jump{g},
\jump{g}
\right\rangle_e
=
0
\qquad
\text{for every }m.
\]
For every face normal $\vec  n_e$, the second-moment condition
implies
\[
\sum_mw_m
\qp{\vec \omega_m\cdot\vec  n_e}^2
=
\frac1d,
\]
so at least one ordinate has
$\vec \omega_m\cdot\vec  n_e\ne0$.  It follows that
\[
\jump{g}=0
\quad\text{on }\mathcal E_h^\circ,
\qquad
g=0
\quad\text{on }\mathcal E_h^\partial.
\]

Substituting
$\widehat\Psi_{m,h}^\epsilon(g)=g$ into
\eqref{eq:proof_normalised_response} and cancelling the collision term
gives
\[
-\sum_{K\in\mathcal T_h}
\qp{g,\vec \omega_m\cdot\nabla v}_K
+
\sum_{e\in\mathcal E_h}
\left\langle
\qp{\vec \omega_m\cdot\vec  n_e}
g_m^{\rm up},
\jump{v}
\right\rangle_e
=
0
\qquad
\forall v\in V_h.
\]
Because $g$ is single-valued on interior faces and has zero trace on
physical boundary faces, elementwise integration by parts reduces this
identity to
\[
\sum_{K\in\mathcal T_h}
\qp{
\vec \omega_m\cdot\nabla g,
v
}_K
=
0
\qquad
\forall v\in V_h.
\]
Since
$\vec \omega_m\cdot\nabla g\in V_h$, this implies
\[
\vec \omega_m\cdot\nabla_hg=0
\qquad
\text{for every }m.
\]
The second-moment condition implies that the ordinate directions span
$\mathbb R^d$, and hence
$\nabla_hg=0$.  Mesh connectedness and continuity across interior
faces show that $g$ is globally constant, while its zero trace on the
nonempty physical boundary gives $g=0$.  Thus the eigenspace of
$\mathsf H_h^\epsilon$ associated with the eigenvalue $1$ is trivial.
Since $\mathsf H_h^\epsilon$ is self-adjoint and
$0\prec\mathsf H_h^\epsilon\preceq I$, this proves
\[
\mathsf H_h^\epsilon\prec I
\]
under homogeneous vacuum inflow.

Using
\[
1-c_\epsilon
=
\epsilon^2
\frac{\overline\sigma_a}{\overline\sigma_t}
\]
in \eqref{eq:exact_correction_form_results} gives
\begin{equation}
\label{eq:proof_normalised_exact_form}
a_{{\rm ex},h}^\epsilon(u,v)
=
\frac1\epsilon
m_t\qp{
\qp{\mathsf H_h^\epsilon}^{-1}u-u,
v
}
+
\epsilon m_a(u,v).
\end{equation}
Since $(\mathsf H_h^\epsilon)^{-1}\succeq I$ and
$m_a(v,v)>0$ for $v\ne0$, this form is symmetric positive definite.
\end{proof}

\begin{proof}[Proof of Proposition~%
\ref{prop:discrete_transport_wellposedness}]
For a solution of
\eqref{eq:discrete_transport_problem}, multiplication of each
directional equation by $\epsilon$ gives
\[
\epsilon
b_{m,h}^\epsilon
\qp{\psi_{m,h}^\star,v}
=
m_t\left(
c_\epsilon\phi_h^\star
+
\frac{\epsilon}{\overline\sigma_t}q_h,
v
\right)
\qquad
\forall v\in V_h.
\]
By the definition of the normalised scalar response,
angular averaging of these directional equations yields
\[
\phi_h^\star
=
\mathsf H_h^\epsilon
\left(
c_\epsilon\phi_h^\star
+
\frac{\epsilon}{\overline\sigma_t}q_h
\right).
\]
Equivalently,
\begin{equation}
\label{eq:discrete_scalar_wellposedness_equation}
\left(
I-c_\epsilon\mathsf H_h^\epsilon
\right)\phi_h^\star
=
\mathsf H_h^\epsilon
\left(
\frac{\epsilon}{\overline\sigma_t}q_h
\right).
\end{equation}

Since
$0\prec\mathsf H_h^\epsilon\preceq I$ and
$0<c_\epsilon<1$, we have
\[
m_t\left(
\left(
I-c_\epsilon\mathsf H_h^\epsilon
\right)v,v
\right)
\ge
(1-c_\epsilon)\Norm{v}_t^2
\qquad
\forall v\in V_h.
\]
Thus
$I-c_\epsilon\mathsf H_h^\epsilon$ is positive definite and
invertible, so
\eqref{eq:discrete_scalar_wellposedness_equation} has a unique solution
$\phi_h^\star$.

For this scalar flux, the well-posedness of each directional transport
problem determines a unique
$\psi_{m,h}^\star\in V_h$.  If
\[
\widetilde\phi_h
:=
\sum_m w_m\psi_{m,h}^\star
\]
denotes their angular average, then the definition of
$\mathsf H_h^\epsilon$ and
\eqref{eq:discrete_scalar_wellposedness_equation} give
\[
\widetilde\phi_h
=
\mathsf H_h^\epsilon
\left(
c_\epsilon\phi_h^\star
+
\frac{\epsilon}{\overline\sigma_t}q_h
\right)
=
\phi_h^\star.
\]
Hence the reconstructed directional fluxes satisfy the required scalar
flux relation and form a solution of the coupled discrete problem.
Uniqueness of the scalar equation and of the directional transport
solves proves uniqueness of the coupled solution.
\end{proof}

\begin{lemma}[Source-iteration contraction]
\label{lem:proof_source_iteration}
The source-iteration error propagator satisfies
\begin{equation}
\label{eq:proof_G_cH}
\mathsf G_h^\epsilon
=
c_\epsilon\mathsf H_h^\epsilon,
\end{equation}
and
\begin{equation}
\label{eq:proof_si_exact_rate}
\Norm{\mathsf G_h^\epsilon}_{\rm ex}
=
\rho\qp{\mathsf G_h^\epsilon}
=
c_\epsilon\lambda_{\max}^\epsilon
\le
c_\epsilon.
\end{equation}
The inequality is strict under homogeneous vacuum inflow. Moreover,
the exact correction identity
\eqref{eq:exact_correction_identity_results} holds.
\end{lemma}

\begin{proof}
Set
\[
\eta_{m,h}^{k+1/2}
:=
\psi_{m,h}^\star-\psi_{m,h}^{k+1/2}.
\]
Subtracting the source-iteration sweep
\eqref{eq:discrete_si_sweep} from the discrete transport equation and
multiplying by $\epsilon$ gives
\[
\epsilon
b_{m,h}^\epsilon
\qp{\eta_{m,h}^{k+1/2},v}
=
m_t\qp{c_\epsilon e_h^k,v}.
\]
Angular averaging proves \eqref{eq:proof_G_cH}. The eigenvalues of $\mathsf H_h^\epsilon$ lie in $(0,1]$.  Therefore
\[
\rho\qp{\mathsf G_h^\epsilon}
=
c_\epsilon\lambda_{\max}^\epsilon
\le
c_\epsilon.
\]
Lemma \ref{lem:proof_normalised_response} shows that the inequality is
strict under homogeneous vacuum inflow.

By \eqref{eq:proof_absorption_ratio} and
\eqref{eq:proof_normalised_exact_form}, the operator representing the
exact energy relative to $m_t$ is
\[
\frac1\epsilon
\left[
\qp{\mathsf H_h^\epsilon}^{-1}-I
\right]
+
\epsilon
\frac{\overline\sigma_a}{\overline\sigma_t}I.
\]
It is therefore a function of the self-adjoint operator
$\mathsf H_h^\epsilon$ and commutes with
$\mathsf G_h^\epsilon=c_\epsilon\mathsf H_h^\epsilon$.
Consequently, $\mathsf G_h^\epsilon$ is self-adjoint and positive in
the exact energy. Its operator norm equals its spectral radius,
proving \eqref{eq:proof_si_exact_rate}.

For $u=e_h^{k+1/2}$,
\eqref{eq:proof_G_cH} gives
\[
\qp{\mathsf H_h^\epsilon}^{-1}e_h^{k+1/2}
=
c_\epsilon e_h^k.
\]
Substitution into \eqref{eq:proof_normalised_exact_form}, together
with
\[
1-c_\epsilon
=
\epsilon^2
\frac{\overline\sigma_a}{\overline\sigma_t},
\]
gives
\[
a_{{\rm ex},h}^\epsilon
\qp{e_h^{k+1/2},v}
=
\qp{
\sigma_s^\epsilon
\qp{e_h^k-e_h^{k+1/2}},
v
}_\Omega,
\]
which is \eqref{eq:exact_correction_identity_results}.
\end{proof}

\subsection{Exact macro--micro factorisation}

Let
\[
\mathbb V_h
:=
V_h^{N_\omega}
\]
with angular collision inner product
\begin{equation}
\label{eq:proof_angular_inner_product}
\qp{U,V}_{\omega,t}
:=
\sum_mw_m m_t(U_m,V_m),
\qquad
\Norm{U}_{\omega,t}^2
:=
\qp{U,U}_{\omega,t}.
\end{equation}
Define the isotropic injection and angular average by
\[
\qp{\mathcal Iu}_m:=u,
\qquad
\mathcal AU:=\sum_mw_mU_m.
\]
Then $\mathcal A=\mathcal I^*$ and
\[
\mathcal Q_0
:=
I-\mathcal I\mathcal A
\]
is the orthogonal projection onto
\[
\mathbb V_h^0
:=
\left\{
U\in\mathbb V_h:\mathcal AU=0
\right\}.
\]
The angular reversal is
\[
\qp{\mathcal RU}_m:=U_{\overline m}.
\]
It is a self-adjoint isometry and commutes with $\mathcal Q_0$.

Define the angular streaming operator
$\mathcal B:\mathbb V_h\to\mathbb V_h$ by
\begin{equation}
\label{eq:proof_directional_B}
m_t\qp{\qp{\mathcal BU}_m,v}
:=
-\sum_K
\qp{U_m,\vec \omega_m\cdot\nabla v}_K
+
\sum_e
\left\langle
\qp{\vec \omega_m\cdot\vec  n_e}
(U_m)^{\rm up},
\jump{v}
\right\rangle_e
\end{equation}
for every $v\in V_h$ and every angular component $m$.
By \eqref{eq:upwind_adjoint_identity},
\begin{equation}
\label{eq:proof_B_reversal_adjoint}
\mathcal B^*
=
\mathcal R\mathcal B\mathcal R.
\end{equation}
On a physical boundary face, the face term in
\eqref{eq:proof_directional_B} uses the zero exterior trace and hence
contains only the outflow contribution.  The reversal identity
\eqref{eq:proof_B_reversal_adjoint} remains valid because reversal
exchanges vacuum inflow and outflow.

The following inverse bound is the point at which the mesh assumptions
enter the macro--micro estimate.

\begin{lemma}[Uniform streaming bound]
\label{lem:proof_uniform_streaming}
Set
\begin{equation}
\label{eq:proof_kappa_h}
\kappa_{h,p}
:=
\max_{K\in\mathcal T_h}
\frac{p^2}{\overline\sigma_t h_K}
=
\frac{\mathrm{Kn}_{h,p}^\epsilon}{\epsilon}.
\end{equation}
There is a constant $C_B>0$ such that
\begin{equation}
\label{eq:proof_B_uniform_bound}
\Norm{\mathcal BU}_{\omega,t}
\le
C_B\kappa_{h,p}\Norm{U}_{\omega,t}
\qquad
\forall U\in\mathbb V_h.
\end{equation}
The constant $C_B$ is independent of $h$, $p$, and the number of
element faces.
\end{lemma}

\begin{proof}
Fix an angular component $m$. For $u,v\in V_h$, the volume term
satisfies
\[
\left|
\sum_K\qp{u,\vec \omega_m\cdot\nabla v}_K
\right|
\le
C_{\rm inv}p^2
\left(
\max_Kh_K^{-1}
\right)
\Norm{u}_{0,\Omega}\Norm{v}_{0,\Omega}.
\]
For the face term, $|\vec \omega_m\cdot\vec  n_e|\le1$,
and the upwind trace on an interior face is one of the two adjacent
element traces, while on a physical boundary face it is either the
unique interior trace or zero. Consequently,
\[
\sum_e\Norm{u_m^{\rm up}}_{0,e}^2
\le
\sum_K\Norm{u}_{0,\partial K}^2,
\qquad
\sum_e\Norm{\jump{v}}_{0,e}^2
\le
2\sum_K\Norm{v}_{0,\partial K}^2.
\]
Cauchy--Schwarz followed by the whole-boundary trace estimate gives
\begin{align*}
&\left|
\sum_e
\left\langle
\qp{\vec \omega_m\cdot\vec  n_e}
u_m^{\rm up},
\jump{v}
\right\rangle_e
\right|
\\
&\qquad\le
\left(
\sum_K\Norm{u}_{0,\partial K}^2
\right)^{1/2}
\left(
2\sum_K\Norm{v}_{0,\partial K}^2
\right)^{1/2}
\\
&\qquad\le
\sqrt2C_{\rm tr}p^2
\left(
\max_Kh_K^{-1}
\right)
\Norm{u}_{0,\Omega}\Norm{v}_{0,\Omega}.
\end{align*}
Since
\[
p^2\max_Kh_K^{-1}
=
\overline\sigma_t\kappa_{h,p},
\]
duality in the $m_t$ inner product gives
\[
\Norm{\qp{\mathcal BU}_m}_t
\le
C_B\kappa_{h,p}\Norm{U_m}_t.
\]
Squaring, multiplying by $w_m$, and summing over $m$ proves
\eqref{eq:proof_B_uniform_bound}.
\end{proof}

Define the mean-zero streaming operator and its micro block by
\begin{equation}
\label{eq:proof_K_and_micro_B}
\mathcal K_h
:=
\mathcal Q_0\mathcal B\mathcal I,
\qquad
\mathcal B_{\mu\mu}
:=
\left.
\mathcal Q_0\mathcal B\mathcal Q_0
\right|_{\mathbb V_h^0},
\qquad
\mathcal D_{\mu\mu}^\epsilon
:=
I+\epsilon\mathcal B_{\mu\mu}.
\end{equation}
For $Z\in\mathbb V_h^0$,
\[
\qp{\mathcal D_{\mu\mu}^\epsilon Z,Z}_{\omega,t}
=
\Norm{Z}_{\omega,t}^2
+
\frac{\epsilon}{2}
\sum_{m,e}w_m
\left\langle
\left|
\vec \omega_m\cdot\vec  n_e
\right|
\jump{Z_m},
\jump{Z_m}
\right\rangle_e
\ge
\Norm{Z}_{\omega,t}^2.
\]
Thus $\mathcal D_{\mu\mu}^\epsilon$ is invertible and
\begin{equation}
\label{eq:proof_micro_resolvent_bound}
\Norm{
\qp{\mathcal D_{\mu\mu}^\epsilon}^{-1}
}_{\omega,t\to\omega,t}
\le1.
\end{equation}

\begin{lemma}[Exact weak scalar factorisation]
\label{lem:proof_exact_factorisation}
Define
\begin{equation}
\label{eq:proof_second_order_transport_form}
\ell_h(u,v)
:=
m_a(u,v)
-
\qp{\mathcal R\mathcal K_hu,\mathcal K_hv}_{\omega,t}.
\end{equation}
Then
\begin{equation}
\label{eq:proof_exact_correction_expansion}
a_{{\rm ex},h}^\epsilon(u,v)
=
j_h(u,v)
+
\epsilon\ell_h(u,v)
+
\epsilon^2r_{h,\epsilon}(u,v),
\end{equation}
where
\begin{equation}
\label{eq:proof_exact_remainder}
r_{h,\epsilon}(u,v)
:=
\qp{
\mathcal R
\qp{\mathcal D_{\mu\mu}^\epsilon}^{-1}
\mathcal B_{\mu\mu}\mathcal K_hu,
\mathcal K_hv
}_{\omega,t}.
\end{equation}
Moreover,
\begin{equation}
\label{eq:proof_remainder_bound}
\norm{r_{h,\epsilon}(u,v)}
\le
C_B\kappa_{h,p}
\Norm{\mathcal K_hu}_{\omega,t}
\Norm{\mathcal K_hv}_{\omega,t}.
\end{equation}
\end{lemma}

\begin{proof}
Let
$g_u^\epsilon=(\mathsf H_h^\epsilon)^{-1}u$, and collect the
directional solutions
$\widehat\Psi_{m,h}^\epsilon(g_u^\epsilon)$ into
$U\in\mathbb V_h$.  Equations
\eqref{eq:proof_normalised_response} and
\eqref{eq:proof_normalised_scalar_response} give
\[
\qp{I+\epsilon\mathcal B}U
=
\mathcal I g_u^\epsilon,
\qquad
\mathcal AU=u.
\]
Write $U=\mathcal Iu+Z$ with $Z\in\mathbb V_h^0$.  Applying
$\mathcal Q_0$ gives
\[
\mathcal D_{\mu\mu}^\epsilon Z
=
-\epsilon\mathcal K_hu,
\qquad
Z
=
-\epsilon
\qp{\mathcal D_{\mu\mu}^\epsilon}^{-1}
\mathcal K_hu.
\]
Applying $\mathcal A$ to the same angular equation gives
\begin{equation}
\label{eq:proof_scalar_schur_identity}
g_u^\epsilon
=
u
+
\epsilon\mathcal A\mathcal B\mathcal Iu
-
\epsilon^2
\mathcal A\mathcal B\mathcal Q_0
\qp{\mathcal D_{\mu\mu}^\epsilon}^{-1}
\mathcal K_hu.
\end{equation}
By \eqref{eq:average_transport_identity} and the definition of
$\mathcal B$,
\[
m_t\qp{\mathcal A\mathcal B\mathcal Iu,v}
=
j_h(u,v).
\]
Furthermore, \eqref{eq:proof_B_reversal_adjoint} and the commutation of
$\mathcal R$ with $\mathcal Q_0$ give
\begin{equation}
\label{eq:proof_off_diagonal_adjoint}
\mathcal A\mathcal B\mathcal Q_0
=
\mathcal K_h^*\mathcal R.
\end{equation}
Substituting \eqref{eq:proof_scalar_schur_identity} into
\eqref{eq:proof_normalised_exact_form} therefore yields the exact
identity
\begin{equation}
\label{eq:proof_exact_schur_form}
a_{{\rm ex},h}^\epsilon(u,v)
=
j_h(u,v)
+
\epsilon m_a(u,v)
-
\epsilon
\qp{
\mathcal R
\qp{\mathcal D_{\mu\mu}^\epsilon}^{-1}
\mathcal K_hu,
\mathcal K_hv
}_{\omega,t}.
\end{equation}
The resolvent identity
\[
\qp{\mathcal D_{\mu\mu}^\epsilon}^{-1}
=
I
-
\epsilon
\qp{\mathcal D_{\mu\mu}^\epsilon}^{-1}
\mathcal B_{\mu\mu}
\]
now proves \eqref{eq:proof_exact_correction_expansion} and
\eqref{eq:proof_exact_remainder}.  Finally,
\eqref{eq:proof_B_uniform_bound} implies
\[
\Norm{\mathcal B_{\mu\mu}}
\le
C_B\kappa_{h,p}.
\]
Combining this with \eqref{eq:proof_micro_resolvent_bound} proves
\eqref{eq:proof_remainder_bound}.
\end{proof}

\subsection{Jump lifting and second-order matching}

Define the broken angular gradient
\begin{equation}
\label{eq:proof_angular_gradient}
\qp{\mathcal C_hu}_m
:=
\overline\sigma_t^{-1}
\vec \omega_m\cdot\nabla_hu.
\end{equation}
The first angular moment in
\eqref{eq:discrete_angular_moments} gives
$\mathcal C_hu\in\mathbb V_h^0$, and the second moment gives
\begin{equation}
\label{eq:proof_angular_gradient_energy}
\qp{\mathcal C_hu,\mathcal C_hv}_{\omega,t}
=
\overline D
\sum_K\qp{\nabla u,\nabla v}_K.
\end{equation}

For $u\in V_h$, define the vector lifting
$\mathcal L_hu\in\mathbb V_h$ by
\begin{align}
\label{eq:proof_lifting_definition}
\qp{\mathcal L_hu,V}_{\omega,t}
:={}&
\sum_{m,e}w_m
\left\langle
\qp{\vec \omega_m\cdot\vec  n_e}
\jump{u},
\avg{V_m}
\right\rangle_e
\nonumber\\
&+
\sum_{m,e}w_m
\left\langle
\left(
\beta_e
-
\frac12
\left|
\vec \omega_m\cdot\vec  n_e
\right|
\right)
\jump{u},
\jump{V_m}
\right\rangle_e
\end{align}
for every $V\in\mathbb V_h$. Taking $V=\mathcal Iz$, using the first
angular moment, and recalling the definition of $\beta_e$ shows that
\[
\mathcal A\mathcal L_hu=0.
\]
Hence $\mathcal L_hu\in\mathbb V_h^0$. On a physical boundary face,
the zero exterior convention gives
\[
\jump{u}=u,
\qquad
\avg{V_m}=\frac12V_m,
\qquad
\jump{V_m}=V_m.
\]
Hence the boundary part of the lifting has coefficient
\[
\frac12
\qp{\vec \omega_m\cdot\vec  n_e}
+
\beta_e
-
\frac12
\left|
\vec \omega_m\cdot\vec  n_e
\right|
=
\beta_e
-
\qp{-\vec \omega_m\cdot\vec  n_e}^{+}.
\]
Its angular average vanishes by central pairing, so
$\mathcal A\mathcal L_hu=0$ continues to hold.

Elementwise integration by parts, together with
\eqref{eq:average_transport_identity}, now gives the exact identity
\begin{equation}
\label{eq:proof_K_gradient_lifting}
\mathcal K_hu
=
\mathcal C_hu-\mathcal L_hu.
\end{equation}

\begin{lemma}[Whole-boundary lifting estimate]
\label{lem:proof_lifting_bound}
There is a constant $C_L>0$, independent of $h$, $p$, and the number
of element faces, such that
\begin{equation}
\label{eq:proof_lifting_bound}
\Norm{\mathcal L_hu}_{\omega,t}^2
\le
C_L\kappa_{h,p}j_h(u,u)
\qquad
\forall u\in V_h.
\end{equation}
Consequently,
\begin{equation}
\label{eq:proof_K_energy_bound}
\Norm{\mathcal K_hu}_{\omega,t}^2
\le
2\overline D\sum_K\Norm{\nabla u}_{0,K}^2
+
2C_L\kappa_{h,p}j_h(u,u).
\end{equation}
\end{lemma}

\begin{proof}
Let $V\in\mathbb V_h$. By
\eqref{eq:proof_lifting_definition},
\begin{align*}
\qp{\mathcal L_hu,V}_{\omega,t}
={}&
\sum_{m,e}w_m
\left\langle
\qp{\vec \omega_m\cdot\vec  n_e}
\jump{u},
\avg{V_m}
\right\rangle_e
\\
&+
\sum_{m,e}w_m
\left\langle
\left(
\beta_e
-
\frac12
\left|
\vec \omega_m\cdot\vec  n_e
\right|
\right)
\jump{u},
\jump{V_m}
\right\rangle_e.
\end{align*}
For the first term, weighted Cauchy--Schwarz and
$\beta_e\ge1/(2d)$ give
\begin{align*}
&\left|
\sum_{m,e}w_m
\left\langle
\qp{\vec \omega_m\cdot\vec  n_e}
\jump{u},
\avg{V_m}
\right\rangle_e
\right|
\\
&\qquad\le
j_h(u,u)^{1/2}
\left(
\sum_{m,e}w_m
\frac{
\qp{\vec \omega_m\cdot\vec  n_e}^2
}{
\beta_e
}
\Norm{\avg{V_m}}_{0,e}^2
\right)^{1/2}
\\
&\qquad\le
C_d\,j_h(u,u)^{1/2}
\left(
\sum_{m,K}w_m
\Norm{V_m}_{0,\partial K}^2
\right)^{1/2}.
\end{align*}
Here
\[
\frac{
\qp{\vec \omega_m\cdot\vec  n_e}^2
}{
\beta_e
}
\le2d.
\]

For the second term, the bounds on $\beta_e$ imply
\[
\frac{
\left|
\beta_e
-
\frac12
\left|
\vec \omega_m\cdot\vec  n_e
\right|
\right|^2
}{
\beta_e
}
\le
\frac d2.
\]
The same argument therefore gives
\begin{align*}
&\left|
\sum_{m,e}w_m
\left\langle
\left(
\beta_e
-
\frac12
\left|
\vec \omega_m\cdot\vec  n_e
\right|
\right)
\jump{u},
\jump{V_m}
\right\rangle_e
\right|
\\
&\qquad\le
C_d\,j_h(u,u)^{1/2}
\left(
\sum_{m,K}w_m
\Norm{V_m}_{0,\partial K}^2
\right)^{1/2}.
\end{align*}
The whole-boundary trace estimate and
\[
p^2h_K^{-1}
\le
\overline\sigma_t\kappa_{h,p}
\]
then imply
\[
\left|
\qp{\mathcal L_hu,V}_{\omega,t}
\right|
\le
C\kappa_{h,p}^{1/2}
j_h(u,u)^{1/2}
\Norm{V}_{\omega,t}.
\]
Taking the supremum over $V$ proves
\eqref{eq:proof_lifting_bound}. Equation
\eqref{eq:proof_K_energy_bound} follows from
\eqref{eq:proof_K_gradient_lifting},
\eqref{eq:proof_angular_gradient_energy}, and
$\Norm{X-Y}^2\le2\Norm{X}^2+2\Norm{Y}^2$.
\end{proof}

When the transport floor is active, define
\begin{align}
\label{eq:proof_scaled_diffusion_form}
d_h(u,v)
:={}&
\sum_K
\left[
\qp{\overline D\nabla u,\nabla v}_K
+
\qp{\overline\sigma_a u,v}_K
\right]
\nonumber\\
&-
\sum_e
\left\langle
\avg{\overline D\nabla u\cdot\vec  n_e},
\jump{v}
\right\rangle_e
\nonumber\\
&-
\sum_e
\left\langle
\avg{\overline D\nabla v\cdot\vec  n_e},
\jump{u}
\right\rangle_e.
\end{align}
Then
\begin{equation}
\label{eq:proof_mip_decomposition}
a_{{\rm MIP},h}^\epsilon(u,v)
=
j_h(u,v)+\epsilon d_h(u,v).
\end{equation}

On $e\in\mathcal E_h^\partial$, the two flux terms in $d_h$ are
\[
-\frac12
\left\langle
\overline D\nabla u\cdot\vec  n_e,
v
\right\rangle_e
-
\frac12
\left\langle
\overline D\nabla v\cdot\vec  n_e,
u
\right\rangle_e.
\]
These are exactly the boundary terms produced by the two
gradient--lifting pairings in the proof of
Lemma~\ref{lem:proof_second_order_matching}.

\begin{lemma}[Exact second-order matching identity]
\label{lem:proof_second_order_matching}
For every $u,v\in V_h$,
\begin{equation}
\label{eq:proof_second_order_matching}
\ell_h(u,v)-d_h(u,v)
=
-\qp{\mathcal R\mathcal L_hu,\mathcal L_hv}_{\omega,t}.
\end{equation}
\end{lemma}

\begin{proof}
Taking $V=\mathcal C_hu$ in
\eqref{eq:proof_lifting_definition} and using the first two angular
moments gives
\begin{equation}
\label{eq:proof_gradient_lifting_pairing}
\qp{\mathcal C_hu,\mathcal L_hv}_{\omega,t}
=
\sum_e
\left\langle
\avg{\overline D\nabla u\cdot\vec  n_e},
\jump{v}
\right\rangle_e.
\end{equation}
Indeed, the first term in the lifting produces the right-hand side.
The contribution containing $\beta_e$ vanishes by the first angular
moment, while the remaining absolute-value contribution vanishes by
central pairing.

Since
\[
\mathcal R\mathcal C_hv=-\mathcal C_hv,
\]
self-adjointness of $\mathcal R$ and
\eqref{eq:proof_gradient_lifting_pairing} give
\begin{equation}
\label{eq:proof_reversed_lifting_pairing}
\qp{\mathcal R\mathcal L_hu,\mathcal C_hv}_{\omega,t}
=
-
\sum_e
\left\langle
\avg{\overline D\nabla v\cdot\vec  n_e},
\jump{u}
\right\rangle_e.
\end{equation}

Substituting
\eqref{eq:proof_K_gradient_lifting} into
\eqref{eq:proof_second_order_transport_form}, and using
\eqref{eq:proof_angular_gradient_energy},
\eqref{eq:proof_gradient_lifting_pairing}, and
\eqref{eq:proof_reversed_lifting_pairing}, yields
\begin{align*}
\ell_h(u,v)
={}&
m_a(u,v)
+
\overline D
\sum_K\qp{\nabla u,\nabla v}_K
\\
&-
\sum_e
\left\langle
\avg{\overline D\nabla u\cdot\vec  n_e},
\jump{v}
\right\rangle_e
\\
&-
\sum_e
\left\langle
\avg{\overline D\nabla v\cdot\vec  n_e},
\jump{u}
\right\rangle_e
\\
&-
\qp{\mathcal R\mathcal L_hu,\mathcal L_hv}_{\omega,t}.
\end{align*}
The first four terms are precisely $d_h(u,v)$, proving
\eqref{eq:proof_second_order_matching}.
\end{proof}

\subsection{Relative comparison of the correction forms}

\begin{lemma}[Mesh- and degree-uniform relative form estimate]
\label{lem:proof_relative_form_estimate}
There are constants $C_{\rm rel}>0$ and
$\mathrm{Kn}_{\rm rel}>0$, independent of $h$, $\epsilon$, $p$, and
the number of element faces, such that, if
\[
\mathrm{Kn}_{h,p}^\epsilon
\le
\min\left\{
\mathrm{Kn}_{\rm floor},
\mathrm{Kn}_{\rm rel}
\right\},
\]
then
\begin{equation}
\label{eq:proof_relative_form_estimate}
\norm{
a_{{\rm ex},h}^\epsilon(u,v)
-a_{{\rm MIP},h}^\epsilon(u,v)
}
\le
C_{\rm rel}\mathrm{Kn}_{h,p}^\epsilon
\Norm{u}_{\rm MIP}
\Norm{v}_{\rm MIP}
\qquad
\forall u,v\in V_h.
\end{equation}
\end{lemma}

\begin{proof}
Choose $\mathrm{Kn}_{\rm rel}\le1$.  Since the transport floor is
active, \eqref{eq:mip_coercivity} gives
\begin{align}
\label{eq:proof_mip_energy_lower_bound}
\Norm{v}_{\rm MIP}^2
\ge{}&
(1-\vartheta)
\left[
\epsilon\overline D
\sum_K\Norm{\nabla v}_{0,K}^2
+j_h(v,v)
\right]
+
\epsilon m_a(v,v).
\end{align}
Equations \eqref{eq:proof_K_energy_bound} and
\eqref{eq:proof_kappa_h} imply
\begin{equation}
\label{eq:proof_scaled_K_bound}
\epsilon
\Norm{\mathcal K_hv}_{\omega,t}^2
\le
C_K\Norm{v}_{\rm MIP}^2,
\end{equation}
provided $\mathrm{Kn}_{h,p}^\epsilon\le1$.

Subtracting \eqref{eq:proof_mip_decomposition} from
\eqref{eq:proof_exact_correction_expansion} gives
\begin{equation}
\label{eq:proof_correction_defect_expansion}
a_{{\rm ex},h}^\epsilon(u,v)
-a_{{\rm MIP},h}^\epsilon(u,v)
=
\epsilon\qp{\ell_h-d_h}(u,v)
+
\epsilon^2r_{h,\epsilon}(u,v).
\end{equation}
The matching identity and lifting estimate give
\begin{align*}
\epsilon
\norm{\qp{\ell_h-d_h}(u,v)}
&\le
\epsilon C_L\kappa_{h,p}
j_h(u,u)^{1/2}j_h(v,v)^{1/2}
\\
&\le
C\mathrm{Kn}_{h,p}^\epsilon
\Norm{u}_{\rm MIP}\Norm{v}_{\rm MIP}.
\end{align*}
Likewise, the remainder estimate and
\eqref{eq:proof_scaled_K_bound} give
\begin{align*}
\epsilon^2
\norm{r_{h,\epsilon}(u,v)}
&\le
C_B\epsilon^2\kappa_{h,p}
\Norm{\mathcal K_hu}_{\omega,t}
\Norm{\mathcal K_hv}_{\omega,t}
\\
&\le
C\mathrm{Kn}_{h,p}^\epsilon
\Norm{u}_{\rm MIP}\Norm{v}_{\rm MIP}.
\end{align*}
Combining these bounds in
\eqref{eq:proof_correction_defect_expansion} proves
\eqref{eq:proof_relative_form_estimate}.
\end{proof}

\subsection{Proof of the main convergence result}

For $u\in V_h$, define $\mathsf P_h^\epsilon u\in V_h$ by
\begin{equation}
\label{eq:proof_correction_map}
a_{{\rm MIP},h}^\epsilon
\qp{\mathsf P_h^\epsilon u,v}
=
a_{{\rm ex},h}^\epsilon(u,v)
\qquad
\forall v\in V_h.
\end{equation}
Equation \eqref{eq:mip_correction_exact_rhs} gives
\[
\delta_h^{k+1}
=
\mathsf P_h^\epsilon e_h^{k+1/2},
\]
and hence
\begin{equation}
\label{eq:proof_dsa_factorisation}
\mathsf E_h^\epsilon
=
\qp{I-\mathsf P_h^\epsilon}
\mathsf G_h^\epsilon.
\end{equation}

\begin{proof}[Proof of Theorem~\ref{thm:main_si_dsa}]
Let $C_{\rm rel}$ and $\mathrm{Kn}_{\rm rel}$ be the constants in
Lemma~\ref{lem:proof_relative_form_estimate}.  Enlarge
$C_{\rm rel}$, if necessary, to a constant $C\ge C_{\rm rel}$ that
also covers the estimates below, and choose
\begin{equation}
\label{eq:proof_Kn0_choice}
\mathrm{Kn}_0
:=
\min\left\{
\mathrm{Kn}_{\rm floor},
\mathrm{Kn}_{\rm rel},
\frac{1}{2C}
\right\}.
\end{equation}

Proposition~\ref{prop:mip_properties} proves consistency, symmetry,
positive definiteness, and activation of the transport floor under
\eqref{eq:main_optically_thick_condition}.  Lemma
\ref{lem:proof_source_iteration} proves
\eqref{eq:main_si_rate}, and Lemma
\ref{lem:proof_relative_form_estimate} proves
\eqref{eq:main_form_comparison}.

It remains to convert the relative form estimate into the DSA
contraction bound. By the definition of $\mathrm{Kn}_0$,
\[
C\mathrm{Kn}_{h,p}^\epsilon
\le
\frac12.
\]
Taking $u=v$ in \eqref{eq:main_form_comparison} gives
\begin{equation}
\label{eq:proof_exact_mip_spectral_equivalence}
\left(
1-C\mathrm{Kn}_{h,p}^\epsilon
\right)
a_{{\rm MIP},h}^\epsilon(v,v)
\le
a_{{\rm ex},h}^\epsilon(v,v)
\le
\left(
1+C\mathrm{Kn}_{h,p}^\epsilon
\right)
a_{{\rm MIP},h}^\epsilon(v,v).
\end{equation}
The map $\mathsf P_h^\epsilon$ is self-adjoint in the MIP inner
product by \eqref{eq:proof_correction_map}. It is also self-adjoint in
the exact energy. Indeed,
\begin{align*}
a_{{\rm ex},h}^\epsilon
\qp{\mathsf P_h^\epsilon u,v}
&=
a_{{\rm MIP},h}^\epsilon
\qp{
\qp{\mathsf P_h^\epsilon}^2u,v
}
\\
&=
a_{{\rm MIP},h}^\epsilon
\qp{
\mathsf P_h^\epsilon u,
\mathsf P_h^\epsilon v
}
\\
&=
a_{{\rm ex},h}^\epsilon
\qp{u,\mathsf P_h^\epsilon v}.
\end{align*}

Since $\mathsf P_h^\epsilon$ is self-adjoint in the MIP inner
product, its eigenvalues are real. If
$\mathsf P_h^\epsilon v=\lambda v$ with $v\ne0$, then
\[
\lambda
=
\frac{
a_{{\rm ex},h}^\epsilon(v,v)
}{
a_{{\rm MIP},h}^\epsilon(v,v)
}.
\]
Therefore
\eqref{eq:proof_exact_mip_spectral_equivalence} implies
\[
\sigma\qp{\mathsf P_h^\epsilon}
\subset
\left[
1-C\mathrm{Kn}_{h,p}^\epsilon,
1+C\mathrm{Kn}_{h,p}^\epsilon
\right].
\]
Self-adjointness in the exact energy now gives
\begin{equation}
\label{eq:proof_I_minus_P_bound}
\Norm{I-\mathsf P_h^\epsilon}_{\rm ex}
\le
C\mathrm{Kn}_{h,p}^\epsilon.
\end{equation}
Combining \eqref{eq:proof_dsa_factorisation},
\eqref{eq:proof_I_minus_P_bound}, and
\eqref{eq:proof_si_exact_rate} proves
\[
\Norm{\mathsf E_h^\epsilon}_{\rm ex}
\le
C\mathrm{Kn}_{h,p}^\epsilon
\rho\qp{\mathsf G_h^\epsilon},
\]
which is \eqref{eq:main_dsa_rate}.

Finally,
$C\mathrm{Kn}_{h,p}^\epsilon\le1/2$ and
$\rho(\mathsf G_h^\epsilon)>0$ give
\[
\rho\qp{\mathsf E_h^\epsilon}
\le
\Norm{\mathsf E_h^\epsilon}_{\rm ex}
\le
\frac12
\rho\qp{\mathsf G_h^\epsilon}
<
\rho\qp{\mathsf G_h^\epsilon},
\]
which proves \eqref{eq:main_strict_comparison} and completes the
proof.
\end{proof}

\section{Numerical verification under vacuum inflow}
\label{sec:numerical_verification}

This section verifies the contraction estimates and
transport-correction comparison established in the preceding analysis.
The experiments consider homogeneous vacuum inflow and examine three
quantities.  The exact source-iteration contraction factor, the
accelerated MIP--DSA contraction factor, and the relative discrepancy
between the exact transport correction and the vacuum-matched MIP
correction.  More extensive computational studies of SIP--DSA and
MIP--DSA on bounded Voronoi meshes are reported in the companion paper
\cite{dsacomp}.

All computations are performed on $\Omega=(0,1)^2.$ For each ordinate
$\vec\omega_m$, homogeneous vacuum inflow is imposed on
\[
\Gamma_m^-
:=
\left\{
x\in\partial\Omega:
\vec\omega_m\cdot\vec n(x)<0
\right\}.
\]
Equivalently, the DG transport discretisation uses the zero exterior
trace on every physical boundary face.

We consider the diffusive scaling
\[
\overline\sigma_t=1,
\qquad
\overline\sigma_a=1,
\qquad
\sigma_t^\epsilon=\epsilon^{-1},
\qquad
\sigma_s^\epsilon=\epsilon^{-1}-\epsilon .
\]
The angular discretisation uses a fixed positive centrally paired
quadrature satisfying
\eqref{eq:discrete_angular_moments}.  Unless stated otherwise,
$N_\omega=16$ uniformly distributed directions are used with weights
$w_m=N_\omega^{-1}$.

Let $\{\varphi_i\}_{i=1}^{\dim V_h}$ denote a basis of the discrete
space.  The exact and MIP correction matrices are defined by
\begin{align}
\left(A_{{\rm ex},h}^{\epsilon}\right)_{ij}
&=
a_{{\rm ex},h}^{\epsilon}
(\varphi_j,\varphi_i),
&
\left(A_{{\rm MIP},h}^{\epsilon}\right)_{ij}
&=
a_{{\rm MIP},h}^{\epsilon}
(\varphi_j,\varphi_i).
\end{align}
The exact correction matrix is assembled from the normalised directional
response problems
\eqref{eq:proof_normalised_response} with homogeneous vacuum inflow,
followed by the construction
\eqref{eq:proof_normalised_scalar_response}
and
\eqref{eq:exact_correction_form_results}.
The MIP correction includes both interior and physical boundary
contributions from
\eqref{eq:vacuum_mip_boundary_contribution}.

The source-iteration and MIP--DSA error propagators are denoted by
$G_h^\epsilon$ and $E_h^\epsilon$, respectively.

%-----------------------------------------------------------------------
\subsection{Experiment 1: Vacuum contraction factors}
\label{sec:numerical_contraction_factors}
%-----------------------------------------------------------------------

The first experiment investigates the source-iteration and accelerated
contraction factors in the diffusive regime.  We fix a centroidal
Voronoi mesh with polynomial degree $p=2$ and vary
\[
  \epsilon\in\{2^{-j}:j=1,\ldots,12\}.
\]
For each value of $\epsilon$ we compute $\rho(G_h^\epsilon), \qquad
\rho(E_h^\epsilon), \qquad \|E_h^\epsilon\|_{\rm ex}, \qquad
\lambda_{\max}^{\epsilon} := \lambda_{\max}(\mathsf H_h^\epsilon).$

The exact operator norm is evaluated from the generalised eigenvalue
problem
\begin{equation}
\left(E_h^\epsilon\right)^{\mathsf T}
A_{{\rm ex},h}^{\epsilon}
E_h^\epsilon x
=
\lambda
A_{{\rm ex},h}^{\epsilon}x ,
\end{equation}
so that
\[
\|E_h^\epsilon\|_{\rm ex}
=
\lambda_{\max}^{1/2}.
\]

For the present coefficients, the scattering ratio is
\[
c_\epsilon
=
\frac{\sigma_s^\epsilon}{\sigma_t^\epsilon}
=
1-\epsilon^2 .
\]
The vacuum source-iteration identity gives
\[
\|G_h^\epsilon\|_{\rm ex}
=
\rho(G_h^\epsilon)
=
c_\epsilon\lambda_{\max}^{\epsilon}
<
c_\epsilon .
\]
Hence the numerical results compare the computed source-iteration
factor against both the exact transport-dependent quantity
$c_\epsilon\lambda_{\max}^{\epsilon}$ and the coefficient-only bound
$c_\epsilon$.

To measure the improvement obtained from diffusion acceleration, we also
record
\begin{equation}
R_{\rm op}^{\epsilon}
:=
\frac{\|E_h^\epsilon\|_{\rm ex}}
{\rho(G_h^\epsilon)},
\qquad
R_{\rm sp}^{\epsilon}
:=
\frac{\rho(E_h^\epsilon)}
{\rho(G_h^\epsilon)} .
\end{equation}
Theorem~\ref{thm:main_si_dsa} predicts that these quantities decay
linearly with the effective cell Knudsen number.

\begin{figure}[h!]
\centering
\includegraphics[width=0.48\textwidth]{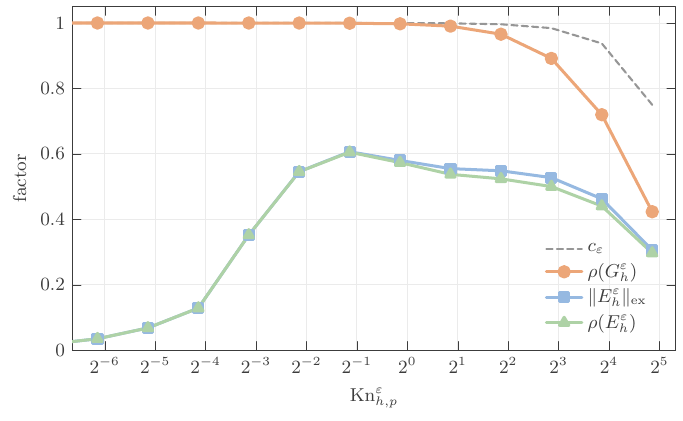}
\hfill
\includegraphics[width=0.48\textwidth]{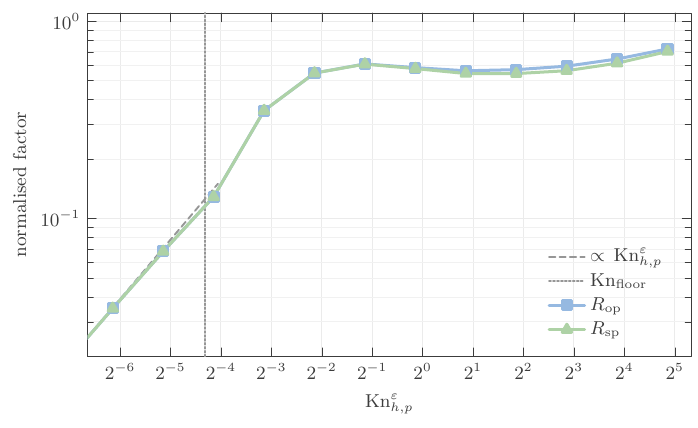}
\caption{
Experiment~1 under homogeneous vacuum inflow.
Left: source-iteration and MIP--DSA contraction factors together with
the exact source-iteration prediction
$c_\epsilon\lambda_{\max}^{\epsilon}$ and the upper bound $c_\epsilon$.
Right: normalised accelerated contraction factors
$R_{\rm op}^{\epsilon}$ and $R_{\rm sp}^{\epsilon}$.
}
\label{fig:numerical_contraction_factors}
\end{figure}

The results in Figure \ref{fig:numerical_contraction_factors} confirm
the exact source-iteration identity to numerical precision and show
the strict improvement obtained from MIP--DSA.  In particular, the
normalised accelerated factors decrease with the effective Knudsen
number, consistent with the estimate of Theorem~\ref{thm:main_si_dsa}.
The spectral radius of the accelerated iteration remains below the
corresponding exact operator norm, as expected from the energy
estimate.

%-----------------------------------------------------------------------
\subsection{Experiment 2: Relative vacuum correction-form estimate}
\label{sec:numerical_form_comparison}
%-----------------------------------------------------------------------

The second experiment examines directly the relative discrepancy between
the exact transport correction and the vacuum-matched MIP correction.
For each mesh, polynomial degree, and value of $\epsilon$, we solve
\begin{equation}
\left(
A_{{\rm ex},h}^{\epsilon}
-
A_{{\rm MIP},h}^{\epsilon}
\right)x
=
\lambda
A_{{\rm MIP},h}^{\epsilon}x
\end{equation}
and define
\begin{equation}
\delta_{h,p}^{\epsilon}
:=
\max |\lambda|.
\end{equation}
Since $A_{{\rm MIP},h}^{\epsilon}$ is positive definite, this quantity
is the relative form discrepancy
\[
\delta_{h,p}^{\epsilon}
=
\sup_{u,v\in V_h\setminus\{0\}}
\frac{
\left|
a_{{\rm ex},h}^{\epsilon}(u,v)
-
a_{{\rm MIP},h}^{\epsilon}(u,v)
\right|
}{
\|u\|_{\rm MIP}\|v\|_{\rm MIP}
}.
\]

We consider two mesh families on $\Omega$, uniform Cartesian meshes
and centroidal Voronoi meshes clipped to the domain boundary.
Polynomial degrees $p\in\{1,2,3,4\}$ are tested, together with several
refinement levels and values of $\epsilon$ chosen to provide
comparable ranges of the effective Knudsen number
\[
\mathrm{Kn}_{h,p}^{\epsilon}
=
\max_{K\in\mathcal T_h}
\frac{p^2}{\sigma_t^\epsilon h_K}.
\]

To identify the asymptotic regime, we also record the facewise floor
indicator
\[
\chi_{h,p}^{\epsilon}
:=
\max_{e\in\mathcal E_h}
\frac{\epsilon\tau_{{\rm SIP},e}}
{\beta_e}.
\]
The transport floor is active when
$\chi_{h,p}^{\epsilon}\le1$, including all physical boundary faces.

\begin{figure}[h!]
\centering
\includegraphics[width=0.48\textwidth]{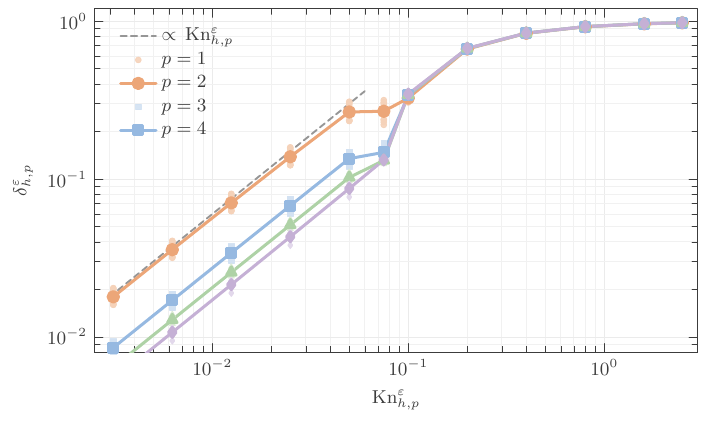}
\hfill
\includegraphics[width=0.48\textwidth]{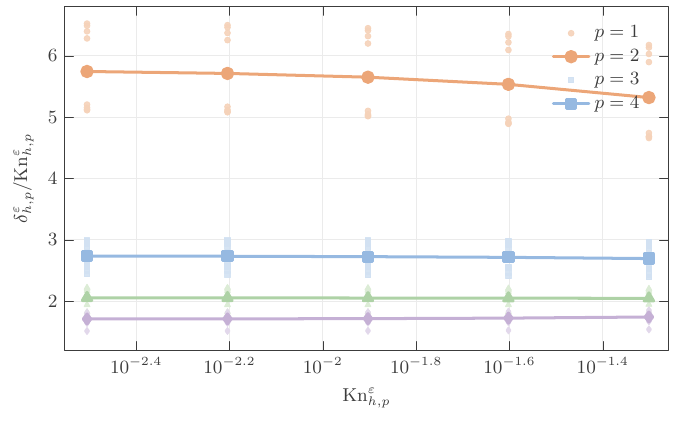}
\caption{
Experiment~2 under homogeneous vacuum inflow.
Left:
relative correction-form discrepancy
$\delta_{h,p}^{\epsilon}$ against the effective Knudsen number for
polynomial degrees $p=1,2,3,4$.
Right:
the scaled quantity
$\delta_{h,p}^{\epsilon}/\mathrm{Kn}_{h,p}^{\epsilon}$.
}
\label{fig:numerical_relative_form_estimate}
\end{figure}

The numerical results in Figure
\ref{fig:numerical_relative_form_estimate} show the predicted linear
scaling of the relative correction-form discrepancy with the effective
Knudsen number.  The scaled quantity remains bounded across polynomial
degrees and mesh families, with comparable behaviour for Cartesian and
Voronoi meshes.  The onset of the linear regime coincides with
activation of the transport floor, including the physical boundary
contributions.

For completeness, we also examine refinement dependence at fixed
small Knudsen number.  The corresponding results for Cartesian and
Voronoi meshes are shown in
Figure~\ref{fig:numerical_refinement_dependence}.

\begin{figure}[h!]
\centering
\includegraphics[width=0.48\textwidth]{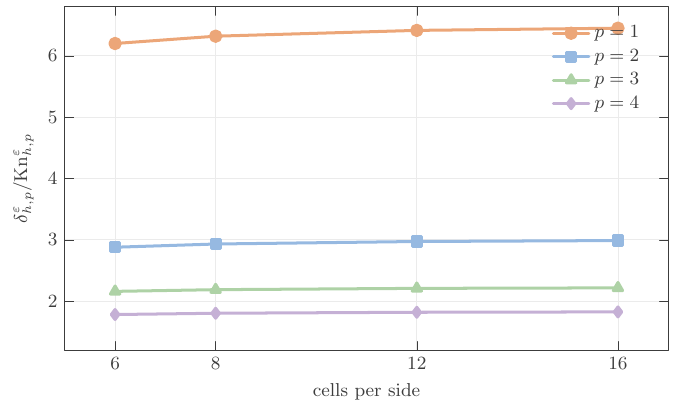}
\hfill
\includegraphics[width=0.48\textwidth]{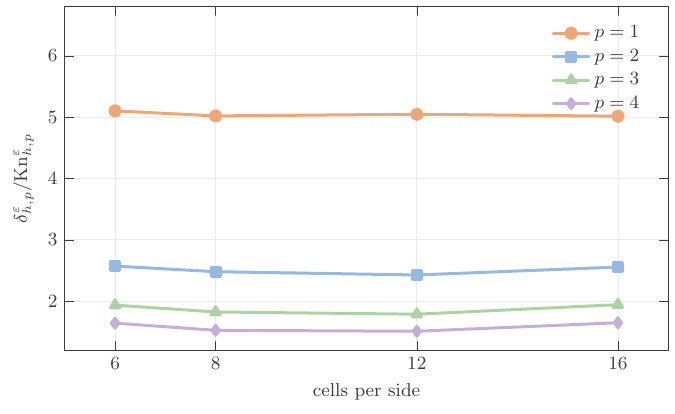}
\caption{
Refinement study for the scaled correction-form discrepancy at fixed
small Knudsen number.
Left: Cartesian meshes.
Right: centroidal Voronoi meshes.
The curves show the dependence of
$\delta_{h,p}^{\epsilon}/\mathrm{Kn}_{h,p}^{\epsilon}$
for polynomial degrees $p=1,2,3,4$.
}
\label{fig:numerical_refinement_dependence}
\end{figure}

\section{Conclusion}

We constructed the exact scalar correction induced by upwind DG
transport solves and compared it directly with a transport-matched MIP
correction.  Homogeneous vacuum inflow is incorporated through a zero
exterior trace on physical boundary faces.  This produces an averaged
boundary leakage term in the exact correction and the corresponding
half-weighted symmetric flux terms in the MIP form.

The macro--micro and lifting identities yield a relative form estimate
controlled by $\mathrm{Kn}_{h,p}^\epsilon$, and hence a contraction
factor for the complete accelerated iteration that is uniform in mesh
size, polynomial degree, and face count under the stated
whole-boundary estimates.  Under vacuum inflow, the exact
source-iteration factor is $c_\epsilon\lambda_{\max}(\mathsf
H_h^\epsilon)$, with the strict reduction arising from boundary
leakage, while the relative MIP--DSA acceleration estimate remains
controlled by the effective cell Knudsen number.

The numerical experiments use the bounded-domain vacuum formulation to
verify the predicted interior operator scaling.  Standard weak
Dirichlet and Marshak diffusion boundary conditions, which differ from
the vacuum-matched form analysed here, are studied computationally in
the companion work \cite{dsacomp}.

\printbibliography
\end{document}